\documentclass[11pt,a4paper]{article}
\usepackage{latexsym,amsfonts,amsmath,amsthm,graphics}
\usepackage{epsfig}
\usepackage{wrapfig}
\usepackage{enumitem}
\usepackage{url}
\usepackage[usenames,dvipsnames,svgnames,table]{xcolor}

\usepackage{algorithmic}
\usepackage{hyperref}

\newtheorem{theorem}{Theorem}
\newtheorem{lemma}{Lemma}
\newtheorem{corollary}{Corollary}
\newtheorem{proposition}{Proposition}

\newtheorem{definition}{Definition}

\theoremstyle{remark}
\newtheorem{example}{Example}
\theoremstyle{definition}
\newtheorem{ass}{Assumption}
\newtheorem{open}{Open Problem}
\newtheorem{remark}{Remark}
\newtheorem*{notation}{Notation}

\newcommand{\norm}[1]{\left\Vert#1\right\Vert}

\newcommand{\bsgamma}{\boldsymbol{\gamma}}

\newcommand{\bsk}{\boldsymbol{k}}

\newcommand{\bsx}{\boldsymbol{x}}

\newcommand{\bsg}{\boldsymbol{g}}

\newcommand{\calF}{{\mathcal{F}}}

\newcommand{\icomp}{\mathtt{i}}

\newcommand{\NN}{\mathbb{N}}
\newcommand{\N}{\mathbb{N}}
\newcommand{\ZZ}{\mathbb{Z}}

\newcommand{\RR}{\mathbb{R}}

\newcommand{\CC}{\mathbb{C}}

\newcommand{\KK}{\mathbb{K}}
\newcommand{\calH}{\mathcal{H}}

\newcommand{\calG}{\mathcal{G}}
\newcommand{\calA}{\mathcal{A}}
\newcommand{\calC}{\mathcal{C}}

\newcommand{\EE}{\mathbb{E}}

\newcommand{\abs}[1]{\left\vert#1\right\vert}

\newcommand{\err}{{\rm err}}

\DeclareMathOperator*{\cost}{cost}
\DeclareMathOperator*{\comp}{comp}
\DeclareMathOperator*{\diam}{diam}

\allowdisplaybreaks

\title{Homogeneous algorithms and solvable problems on cones}

\author{David Krieg and Peter Kritzer}

\date{\today}

\begin{document}

\maketitle

\begin{abstract}
We consider linear problems in the worst case setting.
That is, given a linear operator and a pool of admissible linear measurements, we want to approximate the values of the operator uniformly on a convex and balanced set
by means of algorithms that use at most $n$ such measurements.
It is known that, in general, linear algorithms do not yield an optimal approximation. 
However, as we show in this paper,
an optimal approximation can always be obtained with a homogeneous algorithm.
This is of interest to us for two reasons. First, the homogeneity allows us to extend any error bound on the unit ball to the full input space. 
Second, homogeneous algorithms are better suited to tackle problems on cones, 
a scenario that is far less understood than the classical situation of balls. 
We use the optimality of homogeneous algorithms to prove solvability for a family of problems defined on cones.
We illustrate our results by several examples. 
\end{abstract}

\section{Introduction}

In this paper, we establish the following two findings:
\begin{itemize}
\item[(1)] Homogeneous algorithms are optimal for any linear problem.
\item[(2)] Item (1) can be used to prove solvability of a familiy of problems defined on cones.
\end{itemize}

\subsection{Linear Problems}

We consider problems 
given by a solution operator $S\colon\calF\to \calG$
between normed spaces $\calF$ and $\calG$ over $\KK \in \{\RR,\CC\}$, 
a class $\Lambda \subseteq \KK^{\calF}$ 
of admissible measurements,
and a class $F\subseteq \calF$ of inputs.
We want to approximate the solution $S(f)$ for some unknown $f\in F$
based on the outcome of a finite number $n$ of measurements
$L_1(f),\hdots,L_n(f)$.
The measurements $L_1,\hdots,L_n$ shall be contained in the class 
$\Lambda$ 
and may be chosen adaptively, i.e., 
the choice of $L_i\in\Lambda$ may depend on the already computed
$L_1(f),\hdots,L_{i-1}(f)$.\footnote{In principle, also the number $n=n(f)$ of measurements could be 
chosen adaptively; however, in the setting of $n$th minimal worst-case errors considered in the first three sections, 
such algorithms can be identified with
algorithms using a fixed $n$
(which can be chosen as 
the maximum of all $n(f)$). 
In Section \ref{sec:cones}, the adaptive choice of $n$ will be of importance.}
This results in an \emph{information mapping} of the form
\begin{equation}\label{eq:info_map}
 N\colon \calF\to\KK^n,
 \quad
 N(f) =(L_1(f),\hdots,L_n(f)).
\end{equation}
We consider 
approximations
of the form
\begin{equation}\label{eq:alg_general}
 A_n \colon \calF \to \calG, \quad A_n(f)=\varphi(N(f)), 
\end{equation}
where $\varphi \colon \KK^n \to \calG$ is an arbitrary mapping and often called the \emph{recovery map}.
A mapping of the form \eqref{eq:alg_general} is called an \emph{algorithm}.
If the $L_i$ are not chosen adaptively, i.e.,
$L_1,\hdots,L_n\in\Lambda$ are the same for all $f\in\calF$,
then the algorithm $A_n$ is called non-adaptive.
The error of an algorithm $A_n$ is measured in the norm of $\calG$ and in
the worst case over the given input class $F\subseteq \calF$,
that is,
\begin{equation}\label{def:wc}
 \err(A_n) \,=\,
 \err(A_n,S,F) \,:=\,
 \sup_{f \in F}\, \Vert S(f) - A_n(f) \Vert_\calG. 
\end{equation}
A problem of this form is called a \textit{linear problem} if the following conditions hold.
\begin{samepage}
\begin{itemize}
\item[(1)] The input class $F$ is non-empty, convex (i.e., $\lambda f + (1-\lambda)g \in F$ for all $f,g\in F$ and $\lambda\in (0,1)$), and balanced (i.e., $\lambda f\in F$ for all $f\in F$ and $\lambda\in \KK$ with $|\lambda|\le 1$),
\item[(2)] the solution operator $S\colon\calF \to \calG$ is linear,
\item[(3)] the class $\Lambda$ of admissible measurements contains only linear functionals.
\end{itemize}
\end{samepage}
Typical examples of linear problems are \emph{approximation problems}, where we have $S(f)=f$,
and \emph{integration problems}, where $\calF$ is some class of integrable functions and $S(f)$ is the integral of $f$. 
Examples for the class $\Lambda$ of admissible measurements
are the class $\Lambda^{\rm all} := \calF'$ of all continuous linear functionals (\emph{linear information})
or the class $\Lambda^{\rm std}$ of all function evaluations (\emph{standard information})
if $\calF$ is a function space.
The input class $F$ often equals the unit ball of the space $\calF$ but the model also allows for so-called approximation sets of the form
\[
 F \,=\, \big\{ f\in \calF \,\mid\, {\rm dist}_{\calF}(f,\mathcal{V}) \le \delta\big\},
\]
where $\mathcal{V}$ is a (typically finite-dimensional) subspace of $\cal F$ and $\delta>0$.
We refer to \cite[Section 4.2]{NW08} and the references therein for further details
on the model of linear problems.

\subsection{Homogeneous algorithms for linear problems}

The class of all algorithms of the form \eqref{eq:alg_general}, which we denote by $\mathcal{A}_n$, is quite huge
and contains very complicated and impractical mappings, 
so it is natural to ask whether an (almost) optimal error bound
can already be achieved with simpler algorithms,
obeying a special structure.
One class of particularly simple algorithms are \textit{linear algorithms}
(i.e., algorithms of the form \eqref{eq:alg_general} 
with a non-adaptive and linear information mapping $N$ and a linear recovery map $\varphi$).
It is a classical result in Information-Based Complexity
that linear algorithms are optimal for linear problems in the case
that $\calG=\KK$,
that is, if the solution operator $S$ is a linear functional. Then we have
\begin{equation}\label{eq:linear-optimal}
 \inf_{A_n \in \mathcal{A}_n^{\rm lin}} \err(A_n) \,=\, \inf_{A_n \in \mathcal{A}_n} \err(A_n),
\end{equation}
where $\mathcal{A}_n^{\rm lin}$ denotes the class of all linear algorithms of the form \eqref{eq:alg_general}.
This result goes back to Smolyak \cite{S65} and Bakhvalov \cite{Bak71},
see also \cite{Don94,FP23,Pla96} for similar results in the presence of noise.
Other instances where linear algorithms are optimal for linear problems are 
when $\calF$ is a pre-Hilbert space or when $\calG$ is a space of bounded functions with the sup-norm,
see Math\'e~\cite{M90} and Creutzig and Wojtaszczyk \cite{CW04}.
On the other hand, 
linear algorithms are \textit{not} optimal for all linear problems.
An example where \eqref{eq:linear-optimal} does not hold goes back to Kashin, Garnaev, and Gluskin \cite{GG84,K77}
and is important in the area of compressive sensing. It is given by the approximation problem
\[
 S\colon \ell_1^m \to \ell_2^m, \quad S(f)=f,
\]
on the unit ball $F$ of $\ell_1^m$ with $\Lambda=\Lambda^{\rm all}$.
Here,
non-linear algorithms are \emph{much} better than linear algorithms
if the dimension $m$ is large compared to the number $n$ of measurements.
Indeed, it can be shown that, in this case, the left-hand side of \eqref{eq:linear-optimal} equals
$\sqrt{(m-n)/m}$ 
for all $n<m$, 
while the right-hand side is of order 
$\sqrt{\log(m/n)/n}$.
See, e.g., \cite[VI~Theorem~2.7]{Pinkus} and \cite[Theorem~1.1]{FPRU}.
Furthermore, there exist linear problems for which every linear algorithm has an infinite error, but 
for which 
a nonlinear algorithm with finite error exists, see~\cite{WW86}.

In this paper, we consider the larger class of 
\textit{homogeneous algorithms}.
These are algorithms with the property that $A_n(\lambda f)=\lambda\, A_n(f)$ for all $f\in\calF$ and all $\lambda\in \KK$.
If the same property holds for all real $\lambda\ge 0$, we call an algorithm \textit{positively homogeneous}. 
When linear algorithms are bad, one may hope that at least the homogeneity is reconcilable with a small error, 
and in fact, this is known to be the case 
for various examples 
(see Remark~\ref{rem:famous-hom}).
Here, we show that homogeneous algorithms
are optimal for \emph{all} linear problems up to a factor of at most two.

\begin{theorem}\label{thm:main}
Let 
$(S\colon\calF\to\calG, F,\Lambda)$
be a linear problem and let $\calG$ be complete. For $n\in\N$, we let $\mathcal{A}_n^*$ denote the class of all homogeneous and non-adaptive algorithms of the form \eqref{eq:alg_general}, and, as above, let $\mathcal{A}_n$ denote the more general class of all algorithms of the form \eqref{eq:alg_general}. Then
\begin{equation}\label{eq:thm1}
 \inf_{A_n \in \mathcal{A}_n^*}\err(A_n) \,\le\, 2\,\inf_{A_n \in \mathcal{A}_n} \err(A_n).
\end{equation}
\end{theorem}

\begin{remark}
The infimum on the right-hand side of \eqref{eq:thm1}
is referred to as the \textit{$n$th minimal error} (with respect to all admissible algorithms, where 
the class of admissible algorithms depends on the precise problem specifications). 
\end{remark}

\smallskip

\begin{remark}
In fact, we are going to show that homogeneous recovery maps are optimal for every fixed non-adaptive information mapping $N$,
see Proposition~\ref{prop:main}.
Theorem~\ref{thm:main} is then implied by the known optimality of non-adaptive information mappings, see \cite{GM80,TW80}; 
we also refer to \cite{N96} for a discussion on the power of adaption.
\end{remark}

\begin{remark}
Assuming additional structure,
there is a result by Bartle and Graves \cite{BG}
which implies the existence of a positively homogeneous and continuous approximate spline mapping for every continuous and linear information mapping $N$, see \cite[Theorem~3.3]{DPW}.
The notion of approximate splines is explained in Remark~\ref{rem:spline}.
Using this result, it is easy to obtain a (less general) version of Theorem~\ref{thm:main}.
Here, we take a more general approach,
which is also easier since we disregard continuity.
\end{remark}

\begin{remark}
 We mentioned above that linear algorithms are, in general, not optimal for linear problems, but they are optimal if the target space $\calG$ is the space $B(X)$ of bounded functions on a compact Hausdorff space $X$ (see \cite{CW04}).
It is shown in \cite{P86} that for any normed space $\calG$ there exists a compact Hausdorff space $X$ and a subspace $\widetilde \calG$ of $B(X)$ such that $\calG$ and $\widetilde \calG$ are isometrically isomorphic. Thus, any linear problem $S\colon \calF \to \calG$ may be interpreted as a linear problem $\widetilde S\colon \calF \to \widetilde\calG$, which implies that the $n$th minimal error can be achieved with a linear algorithm $\widetilde A_n\colon \calF \to B(X)$. Thus, in a certain sense, linear algorithms are always optimal for linear problems if we sufficiently blow up the target space. However, since $\widetilde A_n$ maps to $B(X)$ and not to the subspace $\widetilde \calG$, there is in general no meaningful interpretation of the approximation $A_n'(f)$ in terms of the original space of solutions $\calG$.
\end{remark}

\smallskip

There are two reasons why we are interested in results of this type.
First, in the prominent case that
one considers a linear problem where
$F$ is the unit ball of $\calF$, it is easily seen that any (positively) homogeneous algorithm $A_n$ satisfies
\[
 \Vert S(f) - A_n(f) \Vert_\calG \,\le\, \err(A_n) \cdot \Vert f\Vert_\calF
 \quad\text{for all } f\in\calF.
\]
In this sense, a homogeneous algorithm with small error 
is not only good on the unit ball of $\calF$
but instead on the full space.
This is usually not the case for non-homogeneous algorithms.
In particular, as stated in Corollary~\ref{cor:error-notions}, the $n$th minimal error and the complexity of
a linear problem do not change (up to a factor of two) 
if we switch from the error criterion $\err(A_n)$ in \eqref{def:wc}
to the error criterion
\[
 \widetilde\err(A_n) \,:=\, \sup_{f\in\calF\setminus\{0\}} \frac{\Vert S(f) - A_n(f) \Vert_\calG}{\Vert f\Vert_\calF}.
\]
The second reason why we are interested in the results mentioned above is that homogeneous algorithms are better suited for problems that are defined on cones as considered, e.g., in \cite{DHKM20,KNR19},
as we explain in the following subsection.

\subsection{Solvability of problems with inputs lying in a cone}

By a \textit{cone} we generally understand any
subset $\calC$ of a $\KK$-vector space $\calF$ which satisfies that $\lambda f\in\calC$ for all $f\in\calC$ and $\lambda > 0$. Problems on cones are usually not solvable with algorithms that use a fixed number of measurements, see Proposition~\ref{lem:not-solvable}. 
In other words, they are not \textit{uniformly solvable}.
However, they are often solvable with algorithms that use an adaptive number of measurements, 
see \cite{DHKM20,KNR19}. 
We call such problems
\textit{weakly solvable}.
In Theorem~\ref{thm:solvable}, we provide a statement on the solvability of 
problems on a certain family of cones, and this insight
is based on the optimality of homogeneous algorithms for linear problems. We state a short version of this theorem here, in order to give the reader an idea of what we will prove below, but avoiding technical notation. For the full version of the theorem,
the precise definitions of solvability, as well as the
proof, we refer to Section~\ref{sec:cones}.
\begin{theorem}[Short version]
   Let $(S,B_\calF,\Lambda)$ and $(T,B_\calF,\Lambda)$ be uniformly solvable linear problems, where $B_\calF$ is the unit ball of $\calF$, and let $t>0$. Then the problem $(S,\calC_t,\Lambda)$ is weakly solvable, where
\[
 \calC_t \,:=\, \left\{ f\in \calF \colon \Vert f \Vert_\calF \le t\,\Vert Tf \Vert_\calH \right\}.
\]
More precisely, we have the cost bound \eqref{eq:cost_bound_th2}.
\end{theorem}
To illustrate Theorem \ref{thm:solvable} with an example, which is discussed among other examples in Section~\ref{sec:cones}, 
consider the approximation problem $S\colon W_2^1([0,1]) \to L_2(0,1)$ with standard information on the input set
\[
 \calC_t \,:=\, \Big\{ f \in W_2^1([0,1]) 
 \colon \Vert f' \Vert_2 \le t\, \Vert f \Vert_2 \Big\}\, .
\]
Here, there is no algorithm $A_n \in \mathcal A_n$ that uses a fixed number $n$ of function values and has a finite error $\err(A_n) < \infty$. But 
a prescribed error $\varepsilon>0$ can still be guaranteed with a varying finite number of function evaluations. The algorithm, which does not require knowledge of $\Vert f \Vert_2$ or $\Vert f' \Vert_2$, needs at most $\mathcal{O}(t\varepsilon^{-1} \Vert f \Vert_2)$ function values to find an $\varepsilon$-approximation of any $f\in\calC_t$.

\subsection{Structure of the paper and notation}

The rest of the paper is structured as follows. Section~\ref{sec:linear_problems} contains the proof of our first main result, Theorem~\ref{thm:main}. A key step in showing the theorem is Proposition~\ref{prop:main}, which is also proved there. Section~\ref{sec:hom_problems} deals with several follow-up considerations to the previous sections, and with algorithms applied to problems which are not linear, but more generally homogeneous. We also present two open problems. Finally, Section~\ref{sec:cones} considers problems with inputs lying in cones. There, we first outline theoretical foundations, including the proof of Theorem~\ref{thm:solvable}, and then apply the theory to several examples. We conclude the paper
with an example that is left for future research.

\begin{notation}
In the following, depending on the situation, we might sometimes write $Sf$ instead of $S(f)$, $Nf$ instead of $N(f)$, and similarly for other mappings, when it eases readability. 
A mapping $\psi:\calC \to \mathcal{K}$ between two subsets $\calC$ and $\mathcal{K}$ of $\KK$-vector spaces is called (positively) homogeneous if $\psi(\lambda f) = \lambda \psi (f)$ for all $f\in\calC$ and $\lambda\in \KK$ (resp.\ $\lambda \ge 0$) such that $\lambda f \in \calC$.
\end{notation}

\section{Homogeneous algorithms and linear problems}
\label{sec:linear_problems}

The following proposition is at the heart of our findings.
It is a statement on homogeneous solution operators $S\colon \calF \to \calG$ and 
homogeneous information mappings $N\colon \calF \to \KK^n$.
In short, it implies that there always exists a \emph{homogeneous} recovery map $\varphi\colon \KK^n \to \calG$ such that the corresponding algorithm $A_n=\varphi\circ N$ is optimal up to a factor of at most 2.

For a precise statement,
we introduce the diameter of an (information) mapping 
$N \colon \calF \to \KK^n$,
\[
 \diam(N) \,=\,
 \diam(N,S,F) \,:=\,
 \sup \Big\{ \Vert S f - S g \Vert_\calG  \colon f,g\in F,\, Nf=Ng \Big\},
\]
which measures the maximal uncertainty in the solution under the a priori knowledge $f\in F$ if the information about $f$ is given by $N$.
The diameter of information relates to the minimal error that can be achieved with this information by
\begin{equation}\label{eq:diam-err}
 \inf_{\varphi\colon \KK^n \to \calG}\, \err(\varphi \circ N) \,\le\,
 \diam(N)
 \,\le\, 2 \inf_{\varphi\colon \KK^n \to \calG}\, \err(\varphi \circ N).
\end{equation}
The infimum appearing in \eqref{eq:diam-err} is commonly referred to as the radius of information and denoted by $\mathrm{rad} (N)$. We refer to \cite[Chapter 4]{NW08} for a proof of \eqref{eq:diam-err} and for background on these concepts.

\begin{proposition}\label{prop:main}
Let $F$ be a balanced, convex and non-empty subset of a $\KK$-vector space $\calF$, let $\calG$ be a complete normed space,
and let $S\colon \calF \to \calG$ be positively homogeneous.
Then, for every positively homogeneous mapping $N \colon \calF \to \KK^n$ 
and any $\delta>0$, 
there is a positively homogeneous mapping $\varphi^*\colon \KK^n \to \calG$
such that
\begin{equation}\label{eq:main-bound}
 \err(\varphi^* \circ N) \,\le\, (1+\delta)\, \diam(N).
\end{equation}
In particular,
\begin{equation}\label{eq:main-cor}
 \err(\varphi^* \circ N) \,\le\, (2+2\delta)\, \inf_{\varphi\colon \KK^n \to \calG}\, \err(\varphi \circ N).
\end{equation}
Moreover, if $S$ and $N$ are homogeneous,
we can also choose $\varphi^*$ to be homogeneous.
\end{proposition}

\medskip

\begin{remark}\label{rem:spline}
    The main step in the proof will be to define a (positively) homogeneous approximate spline mapping. Here, a mapping $s\colon N(\calF) \to \calF$ is called an \textit{approximate spline}, if it satisfies 
    \begin{itemize}
        \item $N(s(y))=y$ for all $y\in N(\calF)$ and
        \item 
    $\Vert s(Nf)\Vert_F \le (1+\delta) \Vert f\Vert_F$ for all $f\in \calF$,
    \end{itemize} 
    where $\delta>0$ is a given parameter and $\Vert\cdot\Vert_F$ is the Minkowski semi-norm induced by $F$, see~\eqref{eq:Minkowski}.
    It is called an \textit{exact spline} in the case $\delta=0$. 
    Exact splines are studied in \cite[Chapter~4]{TW80}, for example.
    In certain situations, 
    see \cite[Theorem~3.4]{DPW},
    there exists a unique exact spline,
    and then the spline is necessarily (positively) homogeneous.
    But in general, an exact spline does not exist. 
    Here, we show that there always exists a (positively) homogeneous approximate spline. The desired recovery map is then given by $\varphi^* = S \circ s$.\\
    The existence of a positively homogeneous approximate spline was already noted in \cite[Theorem~3.3]{DPW},
    a result that is attributed to Bartle and Graves \cite{BG}.
    This result makes extra assumptions on $\calF$, $N$ and $F$ (e.g., that $N$ is linear and continuous and that $\Vert \cdot\Vert_F$ is a norm), but in return the approximate spline is also continuous.
\end{remark}

\begin{proof}[Proof of Proposition~\ref{prop:main}]
We first observe that the positively homogeneous case implies the homogeneous case.
Indeed, by the first part of the statement, we can find a positively homogeneous mapping $\widetilde\varphi$ with the desired error bound \eqref{eq:main-bound}. We let $M$ be the set of vectors in $\KK^n$ which are either zero or whose first non-zero coordinate is real and positive. Then each $y\in\KK^n\setminus\{0\}$ can be uniquely expressed as $y=\lambda(y) \cdot y_+$ with some $y_+\in M$ and some $\lambda(y)\in\KK$ with $|\lambda(y)|=1$.
We replace the mapping $\widetilde \varphi$ by the mapping $\varphi^*\colon \KK^n \to \calG$ with $\varphi^*(y) := \lambda(y) \cdot \widetilde\varphi(y_+)$. It is easily verified that the new mapping $\varphi^*$ is homogeneous and that the error bound \eqref{eq:main-bound} is preserved under this replacement whenever $S$ and $N$ are homogeneous.

Moreover, Equation \eqref{eq:main-cor} follows from Equation \eqref{eq:main-bound} via \eqref{eq:diam-err}.
Thus our task is to prove the first statement on the existence of a positively homogeneous mapping $\varphi^*$ that satisfies \eqref{eq:main-bound}.

Since $F$ is convex and balanced, it holds that
${\rm span}(F) = \{\lambda f \colon f\in F, \lambda \ge 0\}$. Without loss of generality, we assume that $\calF = {\rm span}(F)$
since the quantities $\err(\varphi^* \circ N)$ and $\diam(N)$ depend on $F$ but not $\calF$.
Note furthermore that the result is trivially correct if $\diam(N) = \infty$. Hence, we may assume for the rest of the proof that $\diam(N) < \infty$.

We define a semi-norm on $\cal F$ by the Minkowski functional
\begin{equation}\label{eq:Minkowski}
 \Vert f \Vert_F \,:=\, \inf \left\{ r>0 \colon f/r \in F \right\}.
\end{equation}
Then
\begin{equation}\label{eq:F_norm_one}
 \{ f \in \calF \colon \Vert f \Vert_F < 1 \}
 \,\subseteq\, F
 \,\subseteq\, \{ f \in \calF \colon \Vert f \Vert_F \le 1 \}.
\end{equation}
Moreover, we define
\[
 K\colon N(\calF) \to [0,\infty), \quad K(y) := \inf_{f \in N^{-1}(y)} \Vert f \Vert_{F} .
\]
This mapping is positively homogeneous;
we have $K(\lambda y) = \lambda K(y)$ for $y\in\KK^n$ and $\lambda \ge 0$.
\medskip

We now define a mapping $\varphi^* \colon N(\calF) \to \calG$.
For this, we split $\calF$ into two disjoint cones, namely 
\[
\calF_0 := \{ f\in \calF \colon K(Nf)=0 \}
\quad\text{and}\quad 
\calF_+ := \{ f\in \calF \colon K(Nf)>0 \}.
\]
That is, $f\in\calF$ belongs to $\calF_0$ if and only if the norm of a function with the same information as $f$ can be
arbitrarily close to zero, and belongs to $\calF_+$ otherwise.
The images of these cones are given by
\[
 N(\calF_0) = \{ y\in N(\calF) \colon K(y)=0 \} 
 \quad\text{and}\quad  
 N(\calF_+) = \{ y\in N(\calF) \colon K(y)>0 \}
\]
and provide a partition of $N(\calF)$.
We will define two positively homogeneous mappings $\varphi_0 \colon N(\calF_0) \to \calG$
and $\varphi_+ \colon N(\calF_+) \to \calG$ such that
\begin{equation}\label{eq1}
 \Vert Sf - \varphi_0(N f) \Vert_{\calG} \,\le\, \diam(N)
 \quad \text{for all } f \in F \cap \calF_0
\end{equation}
and 
\begin{equation}\label{eq2}
 \Vert Sf - \varphi_+(N f) \Vert_{\calG} \,\le\, (1+\delta)\diam(N)
 \quad \text{for all } f \in F \cap \calF_+.
\end{equation}
Then the mapping
\[
 \varphi^* \colon \KK^n \to \calG, \quad \varphi^*(y) = \begin{cases}
       \varphi_0(y) & \mbox{if $y\in N(\calF_0)$,}\\
       \varphi_+(y) & \mbox{if $y\in N(\calF_+)$,}\\
       0 & \mbox{if $y\not\in N(\calF)$,}
      \end{cases}
\]
is positively homogeneous and satisfies \eqref{eq:main-bound}.
Thus, it only remains to show \eqref{eq1} and \eqref{eq2} to complete the proof.

We start with the definition of $\varphi_+$ and the proof of \eqref{eq2}.
Let $y\in N(\calF_+)$.
If $K(y)= 1$, we choose $s(y)$ as an arbitrary element from the pre-image 
$N^{-1}(y)$ that satisfies 
\begin{equation}\label{almost-spline_alt}
 \Vert s(y) \Vert_F \,<\, (1+\delta) K(y).
\end{equation}
For arbitrary $y\in N(\calF_+)$, we have $K(y/K(y))=1$ and put $s(y)=K(y) s(y/K(y))$.
Then the mapping $s\colon N(\calF_+) \to \calF$ is positively homogeneous,
which also implies that \eqref{almost-spline_alt} holds for all $y\in N(\calF_+)$.
Moreover, for all $y\in N(\calF_+)$,
the element $s(y)\in\calF$ satisfies $N(s(y))=y$ 
(i.e., it is an approximate spline).
We put $\varphi_+(y) = S(s(y))$.

The mapping $\varphi_+\colon N(\calF_+) \to \calG$ is positively homogeneous
as a composition of positively homogeneous mappings.
To bound the error of $\varphi_+ \circ N$, 
we note that for all $f\in \calF_+$ with norm $\norm{\cdot}_F$ at most $1/(1+\delta)$,
which implies $f\in F$,
we have $K(Nf)\le 1/(1+\delta)$ and thus by 
\eqref{almost-spline_alt} that $s(Nf)$ is contained in $F$.
Moreover, $f$ and $s(Nf)$ yield the same information.
Therefore,
\[
 \Vert Sf - \varphi_+(Nf) \Vert_\calG \,\le\, \diam(N)
 \quad \text{for all } f \in \calF_+ \text{ with } \Vert f \Vert_{F} \le 1/(1+\delta).
\]
The bound \eqref{eq2} is now obtained by recalling \eqref{eq:F_norm_one} and scaling due to the positive homogeneity of $S$ and $\varphi_+ \circ N$.

We continue with the definition of $\varphi_0$ and the proof of \eqref{eq1}.
To this end, let $y\in N(\calF_0)$ and hence $K(y)=0$.
Then there is a sequence $(f_k) \subseteq N^{-1}(y)$ such that $\Vert f_k \Vert_F \to 0$.
This implies that $(Sf_k)$ is a Cauchy sequence in $\calG$.
Indeed, assume to the contrary that it is no Cauchy sequence.
Then there is some $c>0$ such that for all $n_0$ there are $k,m>n_0$ with $\Vert Sf_k - Sf_m \Vert_\calG \ge c$.
Given any $R>0$, we can choose $n_0$ large enough such that $\Vert f_k\Vert_F$ and $\Vert f_m\Vert_F$ are smaller than $c/R$. Then $h_k = R f_k /c$ and $h_m = R f_m/c$ are in $F$ and satisfy $Nh_k = N h_m$ and $\Vert Sh_k - Sh_m \Vert_\calG \ge R$. This yields $\diam(N)=\infty$, which would be a contradiction.\\
As $\cal G$ is complete, the sequence $(Sf_k)$ has a limit $g$ and we define $\varphi_0(y) := g$.
This definition is independent of the choice of the sequence $(f_k)$.
Indeed, let $(f_k^*)$ be another sequence in $N^{-1}(y)$ with $\Vert f_k^* \Vert_F \to 0$.
Assume to the contrary that $(Sf_k^*)$ has a different limit than $(Sf_k)$, i.e.,
there is some $c>0$ with $\Vert Sf_k - Sf_k^* \Vert_\calG \ge c$ for infinitely many $k$.
For any $R>0$, choosing $k$ large enough, we thus have $h_k := Rf_k/c \in F$ and $h_k^* := Rf_k^*/c \in F$
with $N(h_k)=N(h_k^*)$ and $\Vert Sf_k - Sf_k^* \Vert_\calG \ge R$,
again leading to the contradiction $\diam(N)=\infty$.

It is not hard to check that the mapping $\varphi_0$ is indeed positively homogeneous:
We have $\varphi_0(0)=0$ (since the value of $\varphi_0(0)$ is independent of the concrete choice of the sequence $(f_k) \subseteq N^{-1}(0)$ from above) and for $\lambda>0$ and $y\in N(\calF_0) \setminus \{0\}$,
we know that $\varphi_0(\lambda y)$ may be expressed as the limit of a sequence $(S(\lambda f_k))$,
where $(f_k) \subseteq N^{-1}(y)$ such that $\Vert f_k \Vert_F \to 0$,
and hence $Sf_k \to \varphi_0(y)$; thus $\varphi_0(\lambda y) = \lambda \varphi_0(y)$.
Moreover, for every $y\in N(\calF_0)$, 
we have that $\varphi_0(y)$ is contained in the closure $\overline{S(N^{-1}(y)\cap F)}$.
Thus, for all $f\in N^{-1}(y)\cap F$, we have
\[
 \Vert Sf - \varphi_0(Nf) \Vert_\calG \,\le\, \diam(\overline{S(N^{-1}(y)\cap F)})
 \,=\, \diam(S(N^{-1}(y)\cap F)) \,\le\, \diam(N).
\]
Here, the diameter of a set in $\calG$ is defined in the usual way.
This gives \eqref{eq1}.
\end{proof}

We now turn to the setting of \emph{linear problems}. Here, a classical result from \cite{GM80,TW80} states that the minimal diameter of information can be obtained with non-adaptive information mappings, see, e.g., \cite[Chapter~4.2]{NW08}.
We note that this result is usually stated in the case $\KK=\RR$, but the proof is exactly the same in the case $\KK=\CC$.
Since the proof is short and elegant anyway, let us repeat it.

\begin{lemma}[see \cite{GM80,TW80}]\label{lem:non-ada}
    Let $(S\colon\calF\to\calG, F,\Lambda)$
be a linear problem and $n\in\NN$. For any $\delta>0$, there exists a non-adaptive information mapping $N\colon \calF \to \KK^n$ such that
    \[
     \diam(N) \,\le\, (2+\delta)\,\inf_{A_n \in \mathcal{A}_n} \err(A_n).
    \]
\end{lemma}

\begin{proof}
    Let $A_n \in \mathcal{A}_n$ be an arbitrary algorithm and $N$ be the corresponding (possibly adaptive) information mapping. Let $L_1,\hdots,L_n \in \Lambda$ be the measurement maps that the algorithm chooses for the input $f=0$
    and consider the non-adaptive information mapping $N^{\rm non}:=(L_1,\hdots,L_n)$. Let further $f,g\in F$ with $N^{\rm non}(f)=N^{\rm non}(g)$. Since $F$ is convex and balanced, the function $h = \frac{f-g}{2}$ is contained in $F$, and since $N^{\rm non}$ is linear, we have $N^{\rm non}(h)=0$. But this means that the adaptive information mapping $N$ recursively chooses the same measurement maps as for the zero input, and thus also $N(h)=0$. The same statements hold for the function $-h$, and thus $A_n$ cannot distinguish between $h$ and $-h$; we have $A_n(h)=A_n(-h)$. This gives
    \[
     \err(A_n) \,\ge\, \frac12\, \Vert S(h) - S(-h) \Vert_\calG 
     \,=\, \frac12\, \Vert S(f) - S(g) \Vert_\calG.
    \]
    Taking the supremum over all such $f$ and $g$, we obtain $2\,\err(A_n) \ge \diam(N^{\rm non})$.
\end{proof}

For a linear problem, any non-adaptive information mapping is also homogeneous.
Hence, Proposition~\ref{prop:main} results in Theorem~\ref{thm:main}, the statement of which is repeated here for completeness.

\setcounter{theorem}{0}

\begin{theorem}
Let 
$(S\colon\calF\to\calG, F,\Lambda)$
be a linear problem and let $\calG$ be complete. For $n\in\N$, we let $\mathcal{A}_n^*$ denote the class of all homogeneous and non-adaptive algorithms of the form \eqref{eq:alg_general}, and, as above, let $\mathcal{A}_n$ denote the more general class of all algorithms of the form \eqref{eq:alg_general}. Then
\begin{equation*}
 \inf_{A_n \in \mathcal{A}_n^*}\err(A_n) \,\le\, 2\,\inf_{A_n \in \mathcal{A}_n} \err(A_n).
\end{equation*}
\end{theorem}

\begin{proof}[Proof of Theorem~\ref{thm:main}]
Let $\delta>0$. By 
Lemma~\ref{lem:non-ada}
there exists a non-adaptive (and thus homogeneous) information mapping $N\colon \calF \to \KK^n$ such that
\[
 \diam(N) \,\le\, (2+\delta)\,\inf_{A_n \in \mathcal{A}_n} \err(A_n).
\]
By Proposition~\ref{prop:main}, there exists a homogeneous mapping $\varphi^*\colon\KK^n\to \calG$ such that the homogeneous algorithm $A_n^*:=\varphi^*\circ N$ satisfies
\[
 \err(A_n^*) \,\le\, (1+\delta)\, \diam(N)
 \,\le\, (1+\delta)\, (2+\delta)\,\inf_{A_n \in \mathcal{A}_n} \err(A_n),
\]
and the proof is finished.
\end{proof}

\begin{remark}
 By Theorem \ref{thm:main}, homogeneous recovery maps are essentially (i.e., up to a small absolute multiplicative constant) optimal for linear problems while linear recovery maps are not. To compensate for the non-linearity, one may ask whether homogeneous recovery maps with a finite-dimensional image are optimal. But this is not the case. There exists a linear solution operator $S\colon \calF \to \calG$ such that for any homogeneous information mapping $N$ and every homogeneous recovery map $\varphi$ with finite-dimensional image, the algorithm $\varphi\circ N$ has infinite error, while $\varphi\circ N$ has finite error for other choices of $\varphi$, see \cite[Remark 2.1]{WW86}.
\end{remark}

\begin{remark}\label{rem:famous-hom}
 For many linear problems, as for example in numerical integration, one typically uses algorithms that are not only homogeneous, but even linear.
 So one may ask to what extent
 homogeneous algorithms are relevant
 at all from a practical point of view. 
 Let us thus mention a few specific examples of homogeneous and non-linear algorithms that are used for linear problems in the literature.
 \begin{itemize}
     \item $\ell^1$-minimization or basis pursuit is proven to be vastly superior to any linear method in the area of compressive sensing, see, e.g., the book \cite{FR}.
     \item The median of means is a popular estimator for the computation of integrals. When used in a randomized setting, the median leads to a probability amplification, see, e.g., \cite[Proposition~2.2]{KNR19}. In the deterministic sense, it can be used to obtain certain universality properties, see \cite{GL22,GSM23}.
     \item The paper \cite{Hei23b} considers a randomized setting and gives an example of a linear problem where adaptive algorithms achieve a better rate of convergence than non-adaptive algorithms. The algorithms that achieve the optimal rate are homogeneous, but not linear. 
 \end{itemize}
\end{remark}

\medskip

Let us now assume that $\calF$ is a normed space and that the input class $F$ is the unit ball of~$\calF$.
Then the $n$th minimal worst case error of the problem $(S\colon\calF\to\calG, F,\Lambda)$ is defined as
\[
\err(n) \,:=\, \err(n,S,\Lambda) \,:=\, \inf_{A_n \in \mathcal A_n} \err(A_n,S,F)
\]
with the worst case error $\err(A_n,S,F)$ as defined in \eqref{def:wc}, i.e.,
\[
 \err(A_n) \,=\, \err(A_n,S,F) \,=\, \sup_{\Vert f \Vert_\calF \le 1} \Vert S(f) - A_n(f) \Vert_\calG.
\]
It has recently been shown in \cite[Lemma~1.3]{V22} for a broad class of approximation problems that
the $n$th minimal worst case error over the unit ball 
coincides up to a factor of at most eight\footnote{The formula in \cite[Lemma~1.3]{V22} shows the factor four but only compares with the minimal error of non-adaptive algorithms, which is why an additional factor of two comes into play.} with the $n$th minimal relative error on the full input space, i.e., with
\[
\widetilde\err(n) \,:=\, 
\widetilde\err(n,S,\Lambda) \,:=\,
\inf_{A_n \in \mathcal A_n} \widetilde\err(A_n,S),
\]
where
\[
\widetilde\err(A_n) \,:=\, 
\widetilde\err(A_n,S) \,:=\, 
\sup_{f\in\calF\setminus\{0\}} \frac{\Vert S(f) - A_n(f) \Vert_\calG}{\Vert f\Vert_\calF}.
\]
As a corollary to Theorem~\ref{thm:main}, 
we obtain that this result is true for all linear problems, 
where the factor eight can be replaced by the factor two.

\begin{corollary}\label{cor:error-notions}
Consider a linear problem $(S\colon\calF \to \calG,F,\Lambda)$, where $F$ is the unit ball of $\calF$ and $\calG$ is complete. For any $n\in\N$, we have
\[
 \err(n) \,\le\,
 \widetilde\err(n) \,\le\, 2\, \err(n).
\]
\end{corollary}

\begin{proof}
The second inequality follows from Theorem~\ref{thm:main} and the fact that $\err(A_n) = \widetilde\err(A_n)$ whenever $A_n$ is homogeneous. Namely,
\[
 \widetilde\err(n) \,\le\,
 \inf_{A_n \text{ homogeneous}} \widetilde\err(A_n) \,=\,
  \inf_{A_n \text{ homogeneous}} \err(A_n) \,\le\,
  2 \,\err(n).
\]
The first inequality uses that any algorithm $A_n$ with $A_n(0)=0$ satisfies
\[
 \err(A_n) \,=\, \sup_{f\in\calF \setminus\{0\} \colon \Vert f \Vert_\calF\le 1} \Vert S(f) - A_n (f) \Vert_\calG
 \,\le\, \sup_{f\in\calF \setminus\{0\} \colon \Vert f \Vert_\calF\le 1} \frac{\Vert S (f) - A_n (f) \Vert_\calG}{\Vert f\Vert_\calF}
 \,\le\, \widetilde\err(A_n)
\]
and hence
\[
 \err(n) \,\le\, \inf_{A_n \colon A_n(0)=0} \err(A_n)
 \,\le\, \inf_{A_n \colon A_n(0)=0} \widetilde\err(A_n).
\]
It remains to show that
\begin{equation}\label{eq:error_tilde}
 \inf_{A_n \colon A_n(0)=0} \widetilde\err(A_n)
 \,\le\, \widetilde\err(n).
\end{equation}
To this end, suppose that we have an 
algorithm $A_n$ such that $A_n=\varphi\circ N$ with $\varphi (0) \ne 0$. We then introduce a mapping $\widetilde\varphi$ such that $\widetilde\varphi (0)=0$ and $\widetilde\varphi (y)=\varphi(y)$ for all $y\neq 0$ and set $A_n^*:=\widetilde\varphi \circ N$. 
We claim that $\widetilde\err (A_n^*) \le \widetilde\err (A_n)$, which proves \eqref{eq:error_tilde}.

Indeed, let us consider the set $N^{-1}(0)$ of those $f\in\calF$ for which $N(f)=0$, and for which then $A_n (f)\neq 0$, but $A_n^* (f)=0$. If all such $f$ satisfy that $S(f)=0$, then 
$A_n^*$ is at least as good as $A_n$ and thus $\widetilde\err (A_n^*) \le \widetilde\err (A_n)$. 
If, on the other hand, there exist $f\in N^{-1}(0)$ such that $S(f)\neq 0$, then we argue that $\widetilde{\err}(A_n)=\infty$ and thus $\widetilde\err (A_n^*) \le \widetilde\err (A_n)$ again holds.

Since $S$ is linear, we have $f\ne 0$.
On the other hand,
we have $N(c f)=0$ for any positive constant $c$. Then, however, by the linearity of $S$,
\[
 \widetilde\err (A_n) \,\ge\, \frac{\norm{S(cf)-\varphi (0)}_{\calG}}{\norm{cf}_{\calF}} \,=\, \frac{\norm{c S(f)-\varphi (0)}_{\calG}}{c\norm{f}_{\calF}}.
\]
Letting $c$ tend to zero, we obtain the asserted identity $\widetilde{\err}(A_n)=\infty$.
\end{proof}

\section{Homogeneous algorithms in other settings} \label{sec:hom_problems}

Recall from the introduction that our problems are described by a solution operator $S$ mapping from a vector space $\calF$ to a normed space $\calG$,
an input set $F\subseteq \calF$, and a class $\Lambda\subseteq \KK^\calF$ of admissible measurements. The problem is called linear if
\begin{itemize}
\item[(1)] the worst case error is considered on a non-empty, convex, and balanced subset $F$ of $\calF$, 
\item[(2)] the solution operator $S\colon\calF \to \calG$ is linear,
\item[(3)] the class $\Lambda$ of admissible measurements contains only linear functionals.
\end{itemize}
In Section~\ref{sec:linear_problems},
we have seen that (positively) homogeneous algorithms are optimal for linear problems.
It is natural to ask whether this extends to a more general class of (positively) homo\-geneous problems. Suppose, for example, that
we replace one or several of the above conditions by the weaker conditions,
\begin{itemize}
\item[(1')] the worst case 
is considered on a class $F\subseteq \calF$ satisfying $\lambda f\in F$ for $f\in F$ and $0\le \lambda\le 1$,
\item[(2')] the solution operator $S\colon\calF \to \calG$ is positively homogeneous,
\item[(3')] the class $\Lambda$ of admissible measurements contains only positively homogeneous functionals.
\end{itemize}
Are positively homogeneous algorithms still optimal for such positively homogeneous problems? We do not know the answer to this question.

\smallskip

\begin{open}
    Are positively homogeneous algorithms optimal for positively homogeneous problems?
    That is, if we consider a class of problems where one or several of the conditions (1)--(3) are replaced by their respective weaker variants (1')--(3'), do there exist constants $C>0$ and $b\in \NN$ 
    such that, for all such problems 
    and all $n\in\NN$, we have 
    \[
    \inf_{\substack{ A \in \mathcal{A}_{bn} \\ A \text{ pos.~hom.}}} \err(A)
    \ \le\ 
    C \inf_{A \in \mathcal{A}_n} \err(A) \ \ ?
    \]
\end{open}

Of course, a similar question can be asked for homogeneous problems instead of positively homogeneous problems.
Examples of such positively homogeneous problems could be the computation of some norm or the maximum of a function $f\in F$.

If we only consider positively homogeneous problems where the input set $F$ is convex and balanced,
Proposition~\ref{prop:main} implies that 
positively homogeneous algorithms are optimal 
whenever positively homogeneous information mappings are optimal (in the sense that the corresponding diameter of information is close to minimal). In particular, for such problems, 
positively homogeneous algorithms are optimal
whenever non-adaptive algorithms are optimal.

This leads to the question whether non-adaptive algorithms are optimal for homogeneous problems. Unfortunately, this is not the case.
The optimality of non-adaptive algorithms does not stay true
if any of the conditions (1)--(3) is replaced by its respective weaker variant, as illustrated by the following examples.

\begin{example}
An example where the properties (1'), (2), and (3) are satisfied is given in \cite{Kor94}, 
see also \cite[Section~4.2.1]{NW08}.
Here, $F$ is the class of all monotonically increasing, $\alpha$-H\"{o}lder 
continuous functions on an interval with H\"{o}lder-constant at most one and $\alpha<1$,
$S$ is the embedding into $L_\infty$
and $\Lambda$ is the class of all function evaluations.
Then the error of non-adaptive algorithms is of order $n^{-\alpha}$
while the error of adaptive algorithms is of polynomial order $n^{-1}$,
where $n$ is the number of measurements. We also refer to \cite{GGM23} for related results. 
\end{example}

\begin{example}\label{ex_bisection}
We give an example where the properties (1), (2'), and (3) are satisfied.
We let $F$ be the unit ball in the space $\calF$ of all Lipschitz-continuous functions 
on $[0,1]$ with the norm
\[
 \Vert f \Vert_{\calF} \,:=\, \max \left\{ \Vert f \Vert_\infty\,,\, \sup_{x\ne y} \frac{|f(x)-f(y)|}{|x-y|} \right\}.
\]
We do bisection in order to approximate a point $z\in [0,1/2]$ with $f(z) = \frac12(f(0)+f(1/2))$.
The bisection method converges to a solution $z(f)$ of the above equation.
Note that the limit $z(f)$ of the bisection method is unique
even if the above equation has more than one solution.
The approximation $z_n(f)$ obtained after $n$ bisection steps (using $n+1$ function values)
satisfies $|z(f)-z_n(f)| \le 2^{-n}$.
We consider the solution operator $S(f)=f(z(f)+1/2)$,
which is homogeneous due to $z(\lambda f)=z(f)$ for any $\lambda\ne0$,
and the class $\Lambda^{\rm std}$ of all function evaluations.
The adaptive algorithm $A_n(f)=f(z_n(f)+1/2)$ has an error of at most~$2^{-n}$.

On the other hand,
any non-adaptive algorithm using $n$ function values has an error of at least $1/(8n)$.
Indeed, there is an interval $(a,a+1/(2n))\subseteq [0,1/2]$ that does not contain any sampling point.
The algorithm thus cannot distinguish the two functions $f$ and $g$
defined by the properties that $f(0)=g(0)=0$,
$f$ is constant on $[0,a]$ and $[a+1/(4n),1/2]$,
$g$ is constant on $[0,a+1/(4n)]$ and $[a+1/(2n),1/2]$,
and both functions are linearly increasing with slope one on the rest of the interval $[0,1]$.
We illustrate this situation in Figure \ref{fig_bisection}.
\begin{figure}[h]\label{fig_bisection}
\begin{center}
 \includegraphics[width=0.8\textwidth]{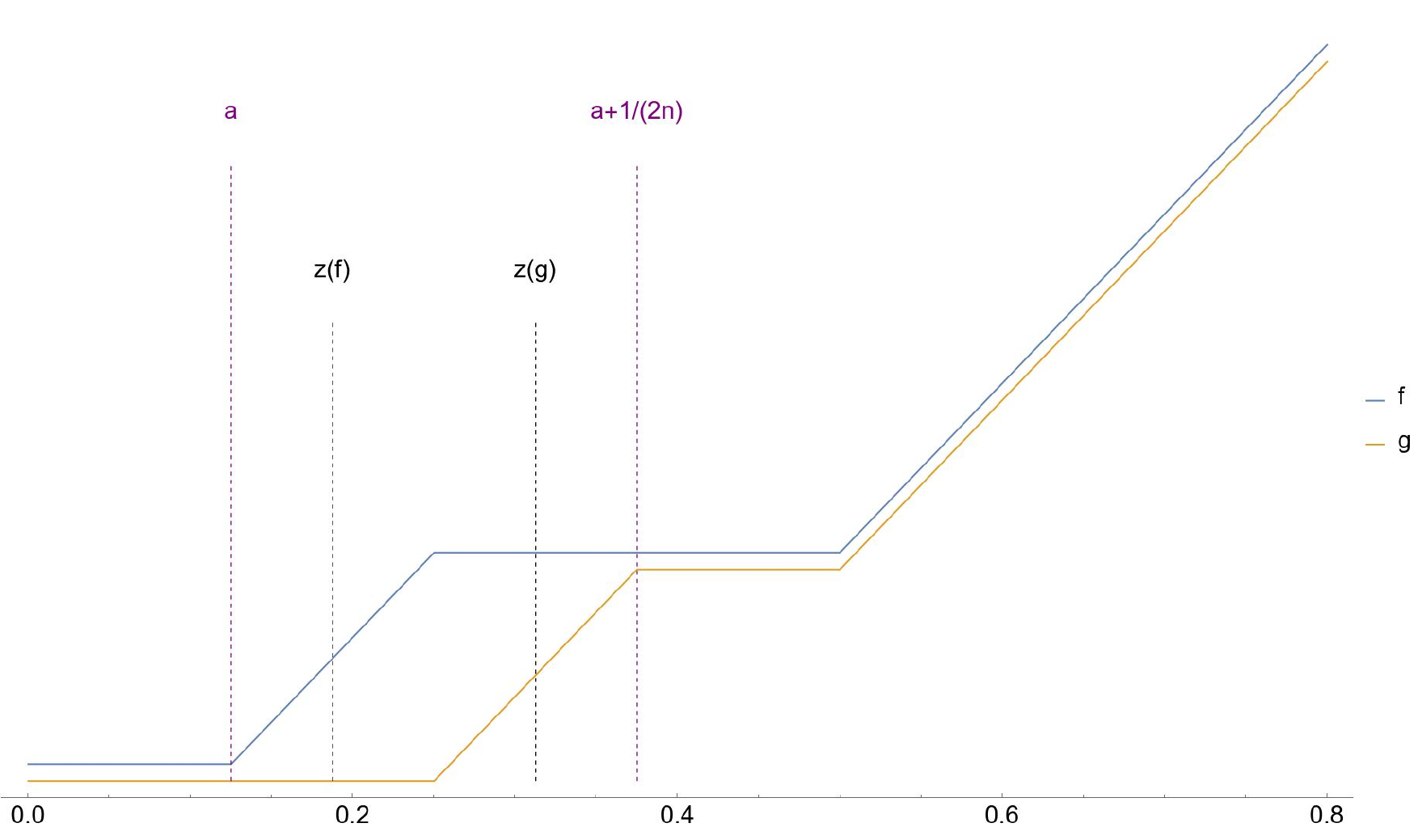} %
 \caption{The situation described in Example \ref{ex_bisection}.}
\end{center} 
\end{figure}
However, since $S(f)$ and $S(g)$ differ by $1/(4n)$,
the error of the algorithm is at least $1/(8n)$.
Thus, the error of certain adaptive algorithms converges exponentially
and the error of non-adaptive algorithms converges only linearly. This problem is clearly related to the zero-finding problem (see, e.g., \cite{S01} and the references therein for detailed information on the zero-finding problem).
\end{example}

\begin{example}
We also give an example where the properties (1), (2), and (3') are satisfied.
We let $\calF$ be defined as in the previous example
and consider the $M$-fold Cartesian product $\calF_M := \calF \times \dots\times \calF$,
with elements denoted by $f=(f_1,\dots,f_M)$,
equipped with the norm
\[
 \Vert f \Vert_{\calF_M} \,:=\, \sum_{i=1}^M \Vert f_i \Vert_{\calF}.
\]
Let $F_M$ be the unit ball of $\calF_M$ and let $\calG_M$ be the same set of functions as $\calF_M$, but with the norm
\[
 \Vert f \Vert_{\calG_M} \,:=\, \max_{i\le M} \Vert f_i \Vert_\infty.
\]
Furthermore, let $S$ be the embedding of $\calF_M$ into $\calG_M$.

The class $\Lambda$ shall consist of the homogeneous measurements 
$f_i(x)$ with $i\le M$ and $x\in [0,1]$ 
and $\Vert f_i \Vert_{\calF}$ with $i\le M$. 
An adaptive algorithm using at most $n$ measurements is given by first computing the $M$ norms $\Vert f_i \Vert_{\calF}$
and then approximating each $f_i$ by a piecewise constant function
that interpolates $m_i$ function values of $f_i$ at equi-spaced points,
where $m_i =0$ 
if $\Vert f_i \Vert_\calF < 2/(n-M)$ (i.e., in this case we approximate $f_i$ by the zero function),
and  $m_i= \lceil \Vert f_i \Vert_\calF (n-M)/2 \rceil \le \Vert f_i \Vert_\calF (n-M)$ otherwise. 
This algorithm has an error of at most $2/(n-M)$.

On the other hand, for a non-adaptive algorithm taking $n$ measurements,
there exists some $i$ where at most $n/M$ samples of $f_i$ are computed (recall the form of $\Lambda$).
Then there is some $f^* \in F$ that equals zero at all those points
but satisfies $\Vert f^*\Vert_\infty \ge M/(2n)$ (by choosing $f^*$ to 
be linear with slope 1 left of the midpoints of the intervals between the interpolation nodes, and slope -1 right of them, respectively). The algorithm cannot distinguish $f\in\calF_M$ defined by $f_j=0$ for $j\ne i$
and $f_i=f^*$ from its negative since all measurements for $f$ and $-f$ yield the same value (including the possible norm-evaluations).
Thus, it has an error of at least $M/(2n)$.
Taking, for instance, $M=n/2$, we obtain that all non-adaptive algorithms
have an error of at least $1/4$, while a suitable adaptive algorithm can achieve an error of at most $4/n$.
\end{example}

\bigskip

Another interesting question is whether the optimality of homogeneous algorithms for linear problems remains true in the randomized setting. Recall from the introduction that our deterministic algorithms are described as the composition of an information mapping $N\colon \calF \to \KK^n$ and a recovery map $\varphi\colon \KK^n \to \calG$. In the randomized setting, $N$ and $\varphi$ can be random and the randomized error of the algorithm $A=\varphi \circ N$ is defined as
\[
 e^{\rm ran}(A) \,:=\, \sup_{f\in F}\, \EE\, \left[\Vert S(f) - A(f) \Vert_\calG \right].
\]
We refrain from a precise description of the setting and refer to \cite[Section~4.3.3]{NW08} and the references therein. We denote by $\mathcal{A}_n^{\rm ran}$ the class of all randomized algorithms with cardinality at most $n$.
The algorithm $A$ is called homogeneous if almost every realization of $A$ is homogeneous.

\smallskip

\begin{open}
    Are homogeneous algorithms optimal for linear problems in the randomized setting?
    That is, do there exist constants $C,c \in\NN$ such that, for all linear problems
    and all $n\in\NN$, we have 
    \[
    \inf_{\substack{ A \in \mathcal{A}_{cn}^{\rm ran} \\ A \text{ homogeneous}}} e^{\rm ran}(A)
    \ \le\ 
    C \inf_{A \in \mathcal{A}_n^{\rm ran}} e^{\rm ran}(A) \ \ ?
    \]
\end{open}

\smallskip

We remark that the paper \cite{Ku16},
see also \cite[Section~2]{KuDiss}, contains a general lower bound for the error of randomized algorithms in terms of Bernstein numbers, while an improved bound is available for homogeneous algorithms.

Again, there is a relation to the question whether non-adaptive algorithms are optimal. But the answer to this question differs from the deterministic setting. It only turned out recently that there exist linear problems where adaptive randomized algorithms are significantly better than non-adaptive ones. The first such example was found by Heinrich in \cite{Hei23a}. We refer to \cite{Hei23b,Hei24,KNW} for further progress on this matter.

\smallskip

As a consequence, the optimality of homogeneous algorithms for linear problems is likely harder to prove in the randomized setting than it was in the deterministic setting (if possible at all). In the deterministic setting, we already knew that we could take $N$ non-adaptive, and hence homogeneous, and only had to prove that we can also take $\varphi$ homogeneous. In the randomized setting, it is already unclear (to us) whether an almost minimal error can be achieved with homogeneous $N$ and arbitrary $\varphi$.


\section{Solvable problems and problems on cones}
\label{sec:cones}

Also in this section, our problems are given by a solution operator $S\colon\calF\to \calG$
between normed spaces $\calF$ and $\calG$ over $\KK$,
a class $\Lambda \subseteq \KK^{\calF}$ 
of admissible measurements,
and a class $F\subseteq \calF$ of inputs.
In the previous sections, we studied linear and homogeneous problems in the worst case setting. Those problems were solvable in the following sense.

\begin{definition}
 We say that the problem $(S\colon \calF \to \calG, F, \Lambda)$ is \emph{uniformly solvable} iff, for any $\varepsilon>0$, there is some $n\in\NN$ and an algorithm $A_n$ of the form \eqref{eq:alg_general} such that $\err(A_n) \le \varepsilon$. In this case, we define the $\varepsilon$-complexity of the problem by
 \[
  \comp(\varepsilon) \,:=\, \comp(\varepsilon,S,F,\Lambda) \,:=\, \min\{n\in\NN \mid \exists A_n\colon \err(A_n) \le \varepsilon \}.
 \]
\end{definition}

We now turn to problems which are not uniformly solvable. Specifically, we are interested in the case that the input set $F\subseteq \calF$ is a cone. Then, in general, the worst case error of any algorithm using homogeneous measurements with fixed $n$ is infinite. 
More precisely, the following holds.

\begin{proposition}\label{lem:not-solvable}
    Let $S\colon\calF\to \calG$ be a 
    homogeneous mapping
    between normed spaces, let $\Lambda$ 
    be a class of homogeneous functionals on $\calF$,
    and let $F\subseteq \calF$ be a cone. Then the problem $(S,F,\Lambda)$ is uniformly solvable if and only if there exists some $n\in\NN$ such that $\err(n)=0$.
\end{proposition}

\begin{proof}
We prove the stronger statement that any algorithm with finite error can be turned into an algorithm with arbitrarily small error and the same cost. Due to \eqref{eq:diam-err}, this is equivalent to saying that every information mapping of the form \eqref{eq:info_map} with finite diameter can be turned into an information mapping with the same cost and a diameter arbitrarily close to zero. To this end, assume that $N$ is an information mapping of the form \eqref{eq:info_map} such that $\diam(N)<\infty$. Given $r>0$, we consider the new information mapping $N_r$, which chooses those measurement maps $L_i$ that the mapping $N$ would choose for $rf$ instead of $f$. Note that the measurement results of $rf$ are available due to $L(rf)=rL(f)$ for all $L\in\Lambda$ and $f\in F$. 
If $f,g\in F$ are such that $N_r(f)=N_r(g)$, then this implies $N(rf)=N(rg)$. Thus,
\[
 \Vert Sf - Sg \Vert_\calG
 \,=\, \frac{1}{r} \Vert S(rf) - S(rg) \Vert_\calG
 \,\le\, \frac{\diam(N)}{r}
\]
and hence the diameter of $N_r$ is bounded by $\diam(N)/r$, which can be made arbitrarily small by choosing $r$ sufficiently large.
\end{proof}

In order to study problems that are not uniformly solvable, we consider a more general class of algorithms, where not only the measurements $L_i\in\Lambda$ themselves, but also the cardinalities of the algorithms, i.e., the number of those measurements, may be chosen adaptively. Roughly speaking this means that, after each measurement, we may decide based on the already computed measurements whether we compute another measurement and which one. Formally, those algorithms can be described as follows. 

An \emph{information mapping} is a function $N\colon F \rightarrow \KK^{<\omega}$, where $\KK^{<\omega}:= \bigcup_{n\in\N} \KK^n$, for which there are functions $L_j: F \times \KK^{j-1} \rightarrow \KK$ with ${L_j(\cdot,y_1,\dots,y_{j-1})\in\Lambda}$ for all $y_1,\dots,y_{j-1} \in\KK$ and a Boolean function $\Delta\colon \KK^{<\omega} \rightarrow \{0,1\}$ such that $N(f) = (y_1,\dots,y_{n(f)})$ with $y_j = L_j(f,y_1,\dots,y_{j-1})$ and $n(f) = \min \{n\in\N \mid \Delta(y_1,\dots,y_n)=0\} $.
That is, if we have already obtained the data $y_1,\dots,y_{j-1}$, we next perform the measurement $L_j(\cdot,y_1,\dots,y_{j-1})$, unless $\Delta(y_1,\dots,y_{j-1})= 0$.

Now, an \emph{algorithm} $A$ for the problem $(S,F,\Lambda)$ 
is defined as a pair $(N,\varphi)$ given by an information mapping $N$ as above 
and an arbitrary mapping $\varphi\colon \KK^{<\omega}\rightarrow \calG$,
which is used to turn our finite information $N(f)\in \KK^{<\omega}$ into an approximation of the solution $S(f) \in \calG$.
This results in a mapping $\varphi \circ N\colon F\to G$ used to approximate the solution operator $S\colon F \to \calG$.
In slight abuse of notation, we denote this mapping also by $A$, that is, we write $A=\varphi\circ N$.
We denote the class of all such algorithms by $\mathcal{A}^+$.
We write $\cost(A,f):=n(f)$ for the 
(information) \emph{cost} of the algorithm at $f\in F$, and for $F_0\subseteq F$, we set
\[
 \cost(A,F_0) \,:=\, \sup \{ \cost(A,f) \mid f\in F_0\}.
\]

We are now ready to define a weaker notion of solvability.

\begin{definition}
    We say that an algorithm $A \in \mathcal{A}^+$ is \emph{$\varepsilon$-approximating} for $S$ on $F$, iff $\Vert Sf - Af \Vert_\calG \le \varepsilon$ for all $f\in F$. The problem $(S\colon \calF \to \calG, F, \Lambda)$ is called \emph{weakly solvable} if, for any $\varepsilon>0$, there exists an $\varepsilon$-approximating algorithm $A \in \mathcal{A}^+$.
\end{definition}

This notion of solvability is inspired by \cite{KNR19}, where an analogous notion of solvability is considered in a randomized setting.
Note that an $\varepsilon$-approximating algorithm has finite cost for each input $f\in F$, but it may have infinite cost on the full input set $F$.

As mentioned above, we are mainly interested in problems where the input set is a cone.
Problems on cones have lately been considered in \cite{KNR19} and \cite{DHKM20}, see also \cite{HJR16, HJRL18}. Here, we study the following general kind of cones. 

\begin{ass}\label{ass:general-cones}
Let $S\colon \calF \to \calG$ be a linear operator between normed spaces $\calF$ and $\calG$ and let $\Lambda$ be a class of linear functionals on $\calF$. 
We consider the problem $(S,\calC_t,\Lambda)$ on the cone
\[
 \calC_t \,:=\, \left\{ f\in \calF \colon \Vert f \Vert_\calF \le t\,\Vert Tf \Vert_\calH \right\},
\]
where $T\colon \calF \to \calH$ is any linear mapping to another normed space $\calH$
and $t>0$ is a constant, which we call the \emph{inflation factor}.
\end{ass}

This setting covers a situation similar to \cite{KNR19},
when $S(f)$ is the integral of a function $f\in L_4$ and $T\colon L_4 \to L_2$ is the identity.
Then, essentially, $\calC_t$ is the cone of all functions with bounded kurtosis.
The setting also covers a situation considered in \cite{DHKM20},
where $S$ is a diagonal operator mapping a Schauder basis of $\calF$ onto a Schauder basis of $\calG$ and $T\colon \calF \to \calF$ is a basis projection with finite rank.
We will come back to these examples later.

The problem $(S,\calC_t,\Lambda)$ as defined in Assumption \ref{ass:general-cones} is included in the assumptions of Proposition~\ref{lem:not-solvable}. Therefore, it is usually not uniformly solvable. 
Here, we show that the problem is weakly solvable for a wide range of operators~$S$ and~$T$. We recall that we write $B_\calF$ to denote the unit ball in the space $\calF$, and by $rB_\calF$ the ball with radius $r$.

\smallskip

\begin{theorem}[Full version]\label{thm:solvable}
	Let Assumption~\ref{ass:general-cones} hold.
    If $(S,B_\calF,\Lambda)$ and $(T,B_\calF,\Lambda)$ are uniformly solvable, then the problem $(S,\calC_t,\Lambda)$ is weakly solvable for any $t>0$. 
    More precisely, for any $\varepsilon>0$, there is an $\varepsilon$-approximating algorithm $A_\varepsilon \in\mathcal{A}^+$ such that, for all $f\in\calC_t\setminus\{0\}$, 
    \begin{equation}\label{eq:cost_bound_th2}
     \cost(A_\varepsilon,f) 
     \,\le\, 
     \comp\left(\frac{1}{5t},T,B_\calF,\Lambda\right) 
     +
     \comp\left(\frac{\varepsilon}{7 t\,\Vert Tf \Vert_\calH} ,S,B_\calF,\Lambda\right).
    \end{equation}
\end{theorem}

\smallskip

To interpret the cost bound, let us assume that the operator $T$ is bounded and that 
$\Vert T \Vert$ and $t$ are absolute constants (and not very big).
Corollary~\ref{cor:error-notions}
implies that $\comp(\widetilde{\varepsilon}/R,S,B_\calF,\Lambda) \le \comp(\widetilde{\varepsilon},S,2RB_\calF,\Lambda)$. 
By putting $c=14t\Vert T\Vert$, we get for sufficiently small $\varepsilon>0$ (assuming that $\comp\left( \varepsilon,S,B_\calF,\Lambda\right) \to \infty$ for $\varepsilon \to 0$) that
\[
 \cost(A_\varepsilon, \calC_t \cap r B_\calF)
 \,\le\, 
 2 \comp\left( \varepsilon,S,crB_\calF,\Lambda\right),
\]
i.e., the cost of the algorithm on a ball is proportional to the complexity of the problem $S$ on a ball of comparable radius.

\begin{proof}
    Let $\varepsilon>0$ and $t>0$. Due to the solvability of $(T,B_\calF,\Lambda)$ and Theorem~\ref{thm:main}, there is a homogeneous algorithm $Q_m\colon \calF \to \calH$ that uses at most $m$ measurements such that 
    \[
     \Vert Tf - Q_mf \Vert_\calH \,\le\, (2t)^{-1} \cdot \Vert f \Vert_\calF, \quad 
     \forall f \in \calF,
    \]
    where $m=\comp(1/(5t),T,B_\calF,\Lambda)$. This implies that, for all $f\in\calC_t$, we have
    \[
     \Vert f \Vert_\calF 
     \,\le\, 		t\, \Vert Tf \Vert_\calH
     \,\le\, 		t\, \Vert Q_m f \Vert_\calH + t\, \Vert Tf - Q_m f \Vert_\calH
     \,\le\, 		t\, \Vert Q_m f \Vert_\calH + \frac12\, \Vert f \Vert_\calF
    \]
    and thus
    \[
     \Vert f \Vert_\calF \,\le\, 2 t\, \Vert Q_m f \Vert_\calH .
    \]
    Our algorithm consists of two steps.
    First, given the unknown $f\in\calC_t$, we compute $Q_mf$.
    If $Q_mf=0$, then we know that $f=0$ and thus obtain the exact solution $S(f)=0$. 
    In case that $Q_mf \ne 0$,
    we adaptively choose the cost $k := \comp(\varepsilon / ((4+2\delta) t \Vert Q_m f \Vert_\calH),S,B_\calF,\Lambda)$ for the second step,
    where $\delta>0$ is arbitrary.
    By the solvability of $(S,B_\calF,\Lambda)$ and Theorem~\ref{thm:main}, there exists a homogeneous algorithm $A_k \colon \calF \to \calG$ that uses at most $k$ measurements, whose worst case error on $B_\calF$ is at most $\varepsilon /(2 t\, \Vert Q_m f \Vert_\calH)$. In particular, for any $f\in \mathcal{C}_t$,
    \[
     \Vert Sf - A_k f \Vert_\calG 
     \,\le\,		\frac{\varepsilon}{2 t\, \Vert Q_m f \Vert_\calH} \cdot \Vert f \Vert_\calF
     \,\le\,		\varepsilon.
    \]
    Thus, our two-step procedure gives an $\varepsilon$-approximating algorithm that uses in total $m+k$ measurements from $\Lambda$.
    Noting that
    \[
     \Vert Q_m f \Vert_\calH \,\le\,  \Vert Tf - Q_m f \Vert_\calH + \Vert T f \Vert_\calH
     \,\le\, \frac{\Vert f \Vert_\calF}{2t} + \Vert Tf \Vert_\calH
     \,\le\, \frac32 \Vert Tf \Vert_\calH,
    \]
    we have
    \[
     k \,\le\, 
     \comp\left(\frac{\varepsilon}{3(2+\delta) t\,\Vert Tf \Vert_\calH} ,S,B_\calF,\Lambda\right),
    \]
    because $\comp(\cdot,S,B_\calF,\Lambda)$ is decreasing in its first argument.
The desired cost bound follows.
\end{proof}

\begin{remark}\label{rem:explicit-algorithms}
In the proof of Theorem \ref{thm:solvable}, we used homogeneous algorithms $Q_m$ and $A_k$ which are optimal in the sense of the complexities of the problems of approximating $T$ and $S$ for certain error thresholds. Since such optimal algorithms are often not known in practice, we point out that the same approach works with arbitrary homogeneous algorithms of our choice. We can take any homogeneous algorithm $Q_m$ for approximating $T$ with arbitrary cost $m$ and worst case error smaller than $1/(2t)$ and then any homogeneous algorithm $A_k$ for $S$ with arbitrary cost $k$ and worst case error smaller than $\varepsilon/(2t\Vert Q_m f \Vert_\calH)$, and take $A_k f$ as an approximation for $Sf$ (or zero, if $Q_mf=0$). This gives us an $\varepsilon$-approximating algorithm on $\calC_t$ with the total cost $m+k$. 
\end{remark}


\medskip

We now discuss several examples.

\subsection{Bounded kurtosis}

Consider the the integration problem
\[
 S\colon \calF \to \RR,
 \quad
 S(f) \,=\, \int_0^1 f(x)\, {\rm d}x,
\]
where $\calF$ shall be the space of all continuous real-valued functions on $[0,1]$, equipped with the 4-norm.
The input set is given by the cone
\[
 \calC_t \,:=\, \{ f\in \calF \colon \Vert f \Vert_4 \le t \, \Vert f \Vert_2 \}
\]
with some $t>1$
and the information is given by the class $\Lambda^{\rm std}$ of function evaluations.

The paper \cite{KNR19}
proves that such problems are weakly solvable in a randomized sense. In our deterministic setting, 
we quickly realize that the problems $S\colon \calF \to \RR$ and $T\colon \calF \to L_2$ are not uniformly solvable on the unit ball of $\calF$ and thus Theorem~\ref{thm:solvable} does not apply.
And indeed, the integration problem on $\calC_t$ is not weakly solvable. 
In fact, there exists no $\varepsilon$-approximating algorithm for any $\varepsilon>0$:

To see this, let $A\in \mathcal{A}^+$ be an arbitrary algorithm and let $\varepsilon>0$ be arbitrary. We let $x_1,\dots,x_n \in [0,1]$ be the measurement points that the algorithm uses for the input zero. Then, for any $\delta\in (0,1)$, one can easily find a continuous function $f\colon [0,1]\to [0,4\varepsilon]$ such that $f$ equals zero at all those points and the Lebesgue measure of the set $f^{-1}(4\varepsilon)$ is at least $1-\delta$.
We get
\[
 \Vert f \Vert_4 
 \,\le\, 4\varepsilon 
 \quad\text{and}\quad
 \Vert f \Vert_2 \,\ge\, 4\varepsilon \cdot (1-\delta)^{1/2}
\]
and thus $f\in \calC_t$ if $\delta$ is sufficiently small. Since $A(f)=A(0)$ and $S(f)-S(0) \ge 4\varepsilon \cdot (1-\delta)$, the algorithm has an error of at least $2\varepsilon (1-\delta)$ for one of the functions $f$ or zero. Thus, the algorithm is not $\varepsilon$-approximating for the integration problem on $\calC_t$.

In the next section, we present a simple example that follows a similar pattern, but which is weakly solvable.

\subsection{Inverse Poincar\'e}

To give another explicit example, let us consider the approximation problem
\[
 S\colon W_2^1([0,1]) \to L_2([0,1]),
 \quad S(f)\,=\,f,
\]
where $W_2^1([0,1])$ is the univariate Sobolev space of all continuous functions $f\colon [0,1]\to \RR$ that possess a weak derivative $f'\in L_2([0,1])$, equipped with the norm $\Vert f\Vert_{W_2^1}^2 = \Vert f\Vert_2^2+ \Vert f'\Vert_2^2$. 
The input set is given by the cone
\[
 \calC_t \,:=\, \left\{ f\in W_2^1([0,1]) \colon \Vert f' \Vert_2 \le t \Vert f \Vert_2 \right\}
\]
of all functions from $\calF$ satisfying a reverse Poincar\'e type inequality with constant $t>0$
and the information is given by the class $\Lambda^{\rm std}$ of all function evaluations.
This situation matches Assumption~\ref{ass:general-cones} if $\calF=W_2^1$, $\calG=\calH=L_2$, and $T=S$, where the cone $\calC_t$ from above is contained in the cone from \eqref{eq:cone_simple} with inflation factor $\sqrt{1+t^2}$. 
We can therefore apply Theorem~\ref{thm:solvable}.

It is a well-known fact that the problem $S$ is uniformly solvable on balls in $\calF$ and that there is a constant $c>0$ such that
\[
 \comp(\varepsilon,S,B_\calF,\Lambda^{\rm std}) \,\le\, c \varepsilon^{-1}
\]
for all $\varepsilon>0$, see, e.g., \cite[Section~4.2.4]{NW08}.
Thus, Theorem~\ref{thm:solvable} yields that the $L_2$-approximation problem is weakly solvable on the full cone $\calC_t$ and gives rise to $\varepsilon$-approximating algorithms $A_\varepsilon$ that satisfy
\[
    \cost(A_\varepsilon,f) 
     \,\le\, 14\, c\, t \cdot \varepsilon^{-1}  \Vert f \Vert_2
    \]
for all $f\in \calC_t$ with $\Vert f\Vert_2 \ge \varepsilon$. 

We conclude this example by noting that also the integration problem is weakly solvable on $\calC_t$. The algorithm $A_\varepsilon^*$ defined by $A_\varepsilon^*(f) := \int_0^1 A_\varepsilon(f)(x)\,{\rm d}x$ satisfies the same cost bound and is an $\varepsilon$-approximating algorithm for the integration problem on $\calC_t$ since
\[
 \bigg| \int_0^1 f(x)\,{\rm d}x - \int_0^1 A_\varepsilon(f)(x)\,{\rm d}x \bigg|
 \,\le\, \int_0^1 |f(x) - A_\varepsilon(f)(x)|\,{\rm d}x
 \,\le\, \big\Vert f - A_\varepsilon(f) \big\Vert_2.
\]

\subsection{Diagonal operators between Banach spaces}\label{sec:diagonal}

We now consider a more general example analogously to \cite{DHKM20}.

Let $\calF$ and $\calG$ be two Banach spaces and let $S\colon \calF \to \calG$ be a continuous linear operator. 
We assume that there exist unconditional Schauder bases $\{u_{\bsk}\}_{\bsk\in\mathbb{I}}$ of $\calF$ 
and $\{v_{\bsk}\}_{\bsk\in\mathbb{I}}$ of $\calG$ such that
\begin{equation}\label{eq:cond_cones_S_basic}
 S (u_{\bsk})=v_{\bsk},\quad \forall\bsk\in\mathbb{I}.
\end{equation}
Here, $\mathbb{I}$ is an arbitrary countable set.
By the spectral theorem, this situation occurs, for example, if $\calF$ and $\calG$ are separable Hilbert spaces and $S$ is a compact and injective linear operator with a dense range. In this case, we can choose $\{u_{\bsk}\}_{\bsk\in\mathbb{I}}$ as an orthonormal basis of eigenfunctions of $S^*S$. But also in the general case, all functions $f\in\calF$ and $g\in\calG$ can be written as a convergent series,
\begin{align*}
    f\,=\,\sum_{\bsk\in\mathbb{I}}  (f,u_{\bsk})_{\calF}\,\cdot u_{\bsk},
    \qquad 
    g\,=\,\sum_{\bsk\in\mathbb{I}}  (g,v_{\bsk})_{\calG}\,\cdot v_{\bsk}, 
\end{align*}
where the coefficients $(f,u_{\bsk})_{\calF}$ and $(g,v_{\bsk})_{\calG}$ are unique.
We have for all $f\in\calF$ that
\[
 S(f) \,=\, \sum_{\bsk\in\mathbb{I}}  (f,u_{\bsk})_{\calF}\,\cdot v_{\bsk}.
\]
Let $\mathbb{I}_0\subseteq \mathbb{I}$ be a finite index set. We denote the span of the functions $u_{\bsk}$, $\bsk\in\mathbb{I}_0$, by $\calF_0$ and the span of the functions $v_{\bsk}$, $\bsk\in\mathbb{I}_0$, by $\calG_0$. Moreover, we consider the basis projections $T\colon \calF \to \calF_0$ and $Q\colon \calG \to \calG_0$, given by
\[
T(f) = \sum_{\bsk\in\mathbb{I}_0} (f,u_{\bsk})_{\calF}\,\cdot u_{\bsk},
\qquad
Q(g) = \sum_{\bsk\in\mathbb{I}_0} (g,v_{\bsk})_{\calG}\,\cdot v_{\bsk},
\]
which are bounded linear operators.
Let further $S_0$ be the restriction of $S$ to $\calF_0$. Consider the cone $\calC_t$ defined analogously to \cite[Equation (7.18)]{DHKM20}, i.e.,
\begin{equation}\label{eq:cone_simple}
 \calC_t:= \left\{ f \in \mathcal F \colon \Vert f \Vert_{\mathcal F} \le t \Vert Tf \Vert_{\mathcal F} \right\}, \quad t\ge 1.
\end{equation}
While this problem was dealt with in \cite{DHKM20} for the case of arbitrary linear information $\Lambda = \Lambda^{\rm all}$, it remained an open question 
how to deal with it if one allows only standard information, 
which seems to be more relevant for practical cases. 
In our approach, any $\Lambda \subseteq \Lambda^{\rm all}$ is allowed,
and we may consider $\Lambda = \Lambda^{\rm std}$ whenever $\calF$ is a space of functions.

The problem matches our Assumption~\ref{ass:general-cones} if we let $\calH=\calF$.
In particular, Theorem~\ref{thm:solvable} applies. That is, the problem is weakly solvable on $\calC_t$, whenever $S$ and $T$ are solvable on the unit ball of $\calF$. We want to argue that the solvability of $S$ implies the solvability of $T$ and simplify the cost bound from Theorem~\ref{thm:solvable}. To this end, let $A_n$ be an algorithm for $S$ with fixed cost $n$. We put $M_n = S_0^{-1} Q A_n$. This is an algorithm for the approximation of $T$ that uses at most $n$ measurements. We have
\[
 \Vert Tf - M_nf \Vert_\calF
 \,\le\, 		\Vert S_0^{-1} \Vert \cdot \Vert STf - Q A_n f \Vert_\calG
  \,\overset{ST=QS}{\le}\, 		\Vert S_0^{-1} \Vert \cdot \Vert Q \Vert \cdot \Vert Sf - A_n f \Vert_\calG.
\]
This means that
\[
 \comp(\varepsilon,T,B_\calF,\Lambda)
 \,\le\, \comp\left( \frac{\varepsilon}{\Vert S_0^{-1} \Vert \cdot \Vert Q \Vert}, S,B_\calF,\Lambda \right).
\]
Summing up, we obtain the following corollary.

\begin{corollary}\label{cor:cost-on-cone}
    Consider the problem $(S,\calC_t,\Lambda)$ as defined in this section.
    If $S$ is uniformly solvable on the unit ball, then $S$ is weakly solvable on $\calC_t$. There exist $\varepsilon$-approximating algorithms $A_\varepsilon$ using information from $\Lambda$ such that, for all $f\in \calC_t$, we have
\begin{equation}\label{eq:cost-on-cone}
    \cost(A_\varepsilon,f) \,\le\,
    2 \cdot \comp\left( \min\left\{ \varepsilon_0\,, \frac{c\cdot \varepsilon}{\Vert f \Vert_\calF} \right\}, S, B_\calF, \Lambda \right),
\end{equation}
where 
\[
 \varepsilon_0 \,:=\, (5t \cdot \Vert S_0^{-1} \Vert \cdot \Vert Q \Vert)^{-1}
 \quad\text{and}\quad
 c \,:=\, (7t\,\Vert T\Vert)^{-1}.
\]
\end{corollary}

Let us see how this formula looks like in the Hilbert space setting mentioned after equation~\eqref{eq:cond_cones_S_basic}, where $S$ is a compact operator between Hilbert spaces. Here, $\Vert T\Vert = \Vert Q \Vert = 1$ and $\Vert S_0^{-1} \Vert$ is the reciprocal of the smallest singular number $\sigma_{\rm min}$ of the mapping $S_0$. Thus, we have
\[
 \cost(A_\varepsilon,f) 
     \,\le\,
     2\, \comp\left( \frac{\varepsilon}{7t \Vert f \Vert_\calF} ,S, B_\calF,\Lambda\right)
\]
for all $0<\varepsilon<\sigma_{\rm min}\cdot \Vert f \Vert_\calF$.

\subsubsection{Simple cones based on a pilot sample}\label{sec:cone_simple}

In the paper \cite{DHKM20}, a setting as in the beginning of Section \ref{sec:diagonal} is considered, but with more particular requirements regarding the norms in $\calF$ and $\calG$. Indeed, let 
us additionally assume that 
\begin{align*}
    \norm{f}_{\calF} &:= \norm{\left(\frac{(f,u_{\bsk})_{\calF}}{\lambda_{\bsk}}\right)_{\bsk\in \mathbb{I}}}_{\rho},\quad \mbox{for some}\ \rho\in [1,\infty],\\
    \norm{g}_{\calG} &:= \norm{\left((g,v_{\bsk})_{\calG}\right)_{\bsk\in \mathbb{I}}}_{\tau},\quad \mbox{for some}\ \tau\in [1,\rho],
\end{align*}
where the $\lambda_{\bsk}$ are positive reals, which we assume to be ordered:
\[
\lambda_{\bsk_1} \ge \lambda_{\bsk_2}\ge \lambda_{\bsk_3}\ge \cdots >0.
\]
In this context, we write $\norm{(\bsx_{\bsk})_{\bsk\in\mathbb{I}}}_{\rho}$ to denote the $\ell_\rho$-norm of the sequence $(\bsx_{\bsk})_{\bsk\in\mathbb{I}}$, and analogously for $\ell_\tau$. 
We again consider $S:\calF\rightarrow\calG$ such that 
$S(u_{\bsk})=v_{\bsk}$ for all $\bsk\in \mathbb{I}$, 
and choose $\mathbb{I}_0$ as the set of integers $\{1,2,\ldots,n_1\}$. Then the cone in \eqref{eq:cone_simple} takes the special form
\[
\calC_t=\left\{f\in\calF\colon \norm{f}_{\calF} \le t \,
\norm{\left(\frac{(f,u_{\bsk_i})_{\calF}}{\lambda_{\bsk_i}}\right)_{1\le i \le n_1}}_{\rho}
\right\}.
\]
This means that the cone $\calC_t$ consists of functions whose norm can be bounded in terms of the pilot sample 
$(f,u_{\bsk_1})_{\calF},\ldots, (f,u_{\bsk_{n_1}})_{\calF}$, 
and the mapping $T$ is given by 
\[
T(f)=\sum_{i=1}^{n_1} (f, u_{\bsk_i})_{\calF}\cdot v_{\bsk_i}.
\]

The paper \cite{DHKM20} then studies an adaptive algorithm $A_n$, defined by
 \[
 A_n (f)=\sum_{i=1}^n (f, u_{\bsk_i})_{\calF}\cdot v_{\bsk_i},
 \]
 where $n=n(f,\varepsilon)$ is allowed to depend on $f\in\calF$ and on an error threshold $\varepsilon>0$, and in general $n$ may be different from $n_1$. The definition of $A_n$
 requires that the information class $\Lambda$ allows access to $(f, u_{\bsk_i})_{\calF}$. This is in general not fulfilled if we consider function spaces and the information class $\Lambda^{\rm std}$,
but for the sake of comparison
let us consider the case that $\Lambda$ contains the measurements 
$(f, u_{\bsk_i})_{\calF}$.

Using the notation in Theorem \ref{thm:solvable} and its proof, we are allowed to approximate $T$ by $Q_m=T$, where $m=n_1$, because we assume that we have access to the $(f, u_{\bsk_i})_{\calF}$. Hence we obviously have $\comp(\varepsilon,T,B_\calF,\Lambda)\le n_1$ for any error threshold $\varepsilon$.

Then Theorem \ref{thm:solvable} or Corollary \ref{cor:cost-on-cone} implies the existence of an 
$\varepsilon$-approximating algorithm $A_{\varepsilon}$ such that 
\[
     \cost(A_\varepsilon,f) 
     \,\le\, 
     \comp\left(\frac{\varepsilon}{7 t\,\Vert Tf \Vert_\calF} ,S,B_\calF,\Lambda\right)
     + n_1 
     .
\]
As outlined in \cite{DHKM20}, we have 
\begin{equation}\label{eq:comp_Hick_et_al}
\comp\left(\frac{\varepsilon}{7 t\,\Vert Tf \Vert_\calF} ,S,B_\calF,\Lambda\right)
=\min\left\{n\in\NN_0\colon \norm{\left((\lambda_{\bsk})_{\bsk\ge n}\right)}_{\rho'} \le \frac{\varepsilon}{7 t\,\Vert Tf \Vert_\calF}\right\},
\end{equation}
where $\rho'$ is such that $1/\rho + 1/\rho'=1/\tau$. In order to determine 
the minimum in \eqref{eq:comp_Hick_et_al}, one needs to find the minimal $n$ such that 
\[
\Vert Tf \Vert_\calF \cdot 
\norm{\left((\lambda_{\bsk})_{\bsk\ge n}\right)}_{\rho'}
=
\norm{\left(\frac{(f,u_{\bsk_i})_{\calF}}{\lambda_{\bsk_i}}\right)_{1\le i \le n_1}}_{\rho} \cdot
\norm{\left((\lambda_{\bsk})_{\bsk\ge n}\right)}_{\rho'}\le \frac{\varepsilon}{7t}.
\]
This corresponds to the analysis done in \cite[Proof of Theorem 7.5]{DHKM20} and leads (up to constants) to the same result. Indeed, it is shown in that paper that the cost is essentially optimal. Furthermore, using these findings, one can then analyze the problem from the viewpoint of Information-Based Complexity, and do a tractability analysis as discussed in \cite{DHKM20}; differences in the constant are due to the fact that the results in \cite{DHKM20} are tailored to this particular situation, whereas Theorem \ref{thm:solvable} and Corollary \ref{cor:cost-on-cone} are formulated in a much more general way.

A drawback of the approach just outlined is that we need information that is 
in general not contained in $\Lambda^{\rm std}$, but we need direct access to the $(f,u_{\bsk_i})_{\calF}$. If this information is not available, then 
one way is to approximate the $(f,u_{\bsk_i})_{\calF}$, possibly by only using information from $\Lambda^{\rm std}$, which may be more practical for applications. In this case, $Q_m$ will in general no longer be equal to $T$ as above, but a proper approximation of $T$. Theorem~\ref{thm:solvable} and Corollary~\ref{cor:cost-on-cone}  enable us to analyze this situation. We present an example of such a situation in the following subsection.

\subsubsection{Simple cones based on a pilot sample: $L_2$-approximation in weighted Korobov spaces using standard information}

We now give a concrete example of the situation described in Section \ref{sec:cone_simple}, but with an algorithm that only uses information from $\Lambda^{\rm std}$.

Let $\calG=L_2 ([0,1]^d)$ and let $\mathcal{F} \subseteq L_2 ([0,1]^d$ be the weighted Korobov space of one-periodic functions on $[0,1]^d$ with smoothness $\alpha>1/2$, i.e.,
\[
\mathcal{F}=\left\{f\colon f(\bsx)=\sum_{\bsk\in \ZZ^d} \widehat{f}(\bsk)e^{2\pi\icomp\bsk\cdot\bsx}, \quad
\norm{f}_{\calF}:= \norm{\left(\widehat{f}(\bsk) \sqrt{r_{2\alpha,\bsgamma} (\bsk)}\right)_{\bsk\in\ZZ^d}}_2 < \infty\right\}.
\]
Here, $\bsgamma=(\gamma_j)_{j\ge 1}$ is a non-increasing sequence of reals in $(0,1]$, and
\[
r_{2\alpha,\bsgamma}(\bsk)=\prod_{j=1}^d r_{2\alpha,\gamma_j} (k_j),\quad \mbox{with}\quad 
r_{2\alpha,\gamma_j}(k_j)=\begin{cases}
                            1 & \mbox{if $k_j=0$,}\\
                            \gamma_j^{-1}\,\abs{k_j}^{2\alpha} & \mbox{otherwise,}
                          \end{cases}
\]
for $\bsk=(k_1,\ldots,k_d)$ (note that we always have $r_{2\alpha,\bsgamma}(\bsk)\ge 1$).
The space $\calF$ is a function class commonly considered in the literature on quasi-Monte Carlo rules, see, e.g., \cite{DKP22}.

Let $S(f)=f$, i.e., $S$ is the embedding of $\calF$ in $\calG$.
This means that in the notation of the previous section we have $\rho=\tau=2$, $\mathbb{I}=\ZZ^d$, $u_{\bsk}=v_{\bsk}=e^{2\pi\icomp\bsk\cdot \circ}$, and 
\[
\lambda_{\bsk}= (r_{2\alpha,\bsgamma} (\bsk))^{-1/2},\quad \mbox{for}\ \bsk\in\ZZ^d.
\]
For a constant $M>1$, let
\[
  \mathbb{I}_0=\mathcal{A}_{d,M}:= \{\bsk\in\ZZ^d\colon r_{2\alpha,\bsgamma} (\bsk)\le M\},
\]
so $\mathcal{A}_{d,M}$ contains the indices $\bsk$ corresponding to the ``largest'' Fourier coefficients of a given $f\in\mathcal{F}$, or putting it differently, $\mathcal{A}_{d,M}$ contains the collection of those $\abs{\mathcal{A}_{d,M}}$ indices $\bsk$ for which the values of $\lambda_{\bsk}=\sqrt{r_{2\alpha,\bsgamma}(\bsk)}$ are the largest.
In the notation of Section \ref{sec:cone_simple}, $\abs{\mathcal{A}_{d,M}}=n_1$.

Let the cone $\calC_t$ be defined as in \eqref{eq:cone_simple}, which in this case yields
\[
\calC_t=\left\{f\in\calF\colon \norm{f}_{\calF} \le t \,
\norm{\left(\widehat{f}(\bsk) \sqrt{r_{2\alpha,\bsgamma}(\bsk)}\right)_{\bsk\in \mathcal{A}_{d,M}}}_{2}
\right\}.
\]
Let also the mappings $T$, $Q$, and $S_0$ be defined as in the beginning of Section \ref{sec:diagonal}.

As shown in \cite{DHKM20}, an---in the sense of the cost---optimal 
algorithm using information from $\Lambda^{\rm all}$ would be to approximate $f\in\calC_t$ by the truncated Fourier series
\[
A_n (f)=\sum_{\bsk \in \calA_{d,M'}} \widehat{f}(\bsk) e^{2\pi\icomp \bsk\cdot\bsx},
\]
with $M'=M'(\varepsilon,f)$ chosen adaptively, and $\calA_{d,M'}$ defined anaolously to $\calA_{d,M}$
which exactly corresponds to the algorithm $A_n$ in Section \ref{sec:cone_simple}. 

However, we would now like to take a different route, where we only use information from $\Lambda^{\rm std}$. Recall that it is described in Remark~\ref{rem:explicit-algorithms} how we can find an $\varepsilon$-approximating algorithm on the cone $\calC_t$ using only information from $\Lambda^{\rm std}$.
The algorithm consists of two successive parts:
\begin{itemize}
    \item a homogeneous algorithm $Q_m$ for approximating $T$ whose worst case error on the unit ball is smaller than the constant $1/(2t)$,
    \item and a homogeneous algorithm $A_k$ for approximating $S$ whose worst case error on the unit ball is smaller than $\varepsilon/(2t\Vert Q_m f \Vert_\calH)$.
\end{itemize}

For the first part, we can consider an algorithm 
$Q_m=\widetilde{A}_m$ that uses an $N$-point rank-1 lattice rule to approximate Fourier coefficients of $f\in \calC_t$, and then replaces the $\widehat{f}(\bsk)$ by these approximations.
Such algorithms were analyzed in, e.g., \cite{KSW06}, see also \cite{DKP22} for an overview. The integration nodes used in the lattice rule are given by
\[
\bsx_{\ell}=\left\{\frac{\ell\bsg}{N}\right\},\quad 0\le \ell\le N-1,
\]
where $\{y\}=y-\lfloor y \rfloor$ denotes the fractional part of a real $y$, and where $\bsg$ is a suitably chosen \textit{generating vector} of the lattice rule. Further details on the choice of such $\bsg$ can be found again in \cite{KSW06} and \cite{DKP22}. Then, $\widetilde{A}_m$ is
defined by
\begin{equation}\label{eq:lattice}
(\widetilde{A}_m (f))(\bsx):=
\sum_{\bsk\in\mathcal{A}_{d,M}} 
\left(\frac{1}{N}\sum_{\ell=0}^{N-1} f\left(\left\{\frac{\ell\bsg}{N}\right\}\right)e^{-2\pi\icomp \ell \bsk\cdot\bsg/N}\right)e^{2\pi\icomp \bsk\cdot\bsx}
\end{equation}
for $f\in\calC_t$ and $\bsx\in [0,1]^d$. Note that $\widetilde{A}_m$ is linear, and in particular homogeneous. Moreover, $\widetilde{A}_m$ uses 
$m=N$ function evaluations as information measurements.

For the algorithm $A_k$ in the second part, one can again use an algorithm of the form \eqref{eq:lattice}
with index set $\mathcal{A}_{d,M'}$ instead of $\mathcal{A}_{d,M}$ and lattice size $k$ instead of $N$, where $M'$ and $k$ are adapted to the outcome of $\widetilde{A}_m(f)$.
Alternatively, we can use a least-squares estimator of the form 
\[
 A_k(f) \,=\, \underset{g\in \mathcal{T}(\mathcal{A}_{d,M'})}{\rm argmin} \, \sum_{i=1}^k\left| f(x_i) - g(x_i)\right|^2.
\]
Here, $\mathcal{T}(I)$ denotes the space of all trigonometric polynomials with frequencies from $I\subseteq \ZZ^d$.
This way, the information cost $k$ will be much lower than with algorithms of the form \eqref{eq:lattice}, 
see \cite{DKU,KU1,NSU} for upper bounds using least-squares
and \cite{BKUV,S60} for lower bounds using lattices. 
The points $x_i$ of the least-squares algorithm can simply be chosen as realizations of i.i.d.\ uniformly distributed random variables on $[0,1]^d$, see \cite{KU1}. 
The drawback of this construction is that it only works with high probability
and there is always a (very) small probability
that the worst-case error of the resulting algorithm $A_k$ will be larger than the desired threshold $\varepsilon/(2t\Vert Q_m f \Vert_\calH)$.

No matter how we decide,
we end up with
an $\varepsilon$-approximating algorithm using information from $\Lambda^{\rm std}$ such that \eqref{eq:cost-on-cone} holds, with the same notation as in Corollary~\ref{cor:cost-on-cone}. We stress that this is an advancement in comparison to the results in \cite{DHKM20}, where there were no tools available to analyze 
algorithms and complexity for information from $\Lambda^{\rm std}$.

\subsubsection{Outlook on future work: Multi-layer cones}\label{sec:cone_not_simple}

As mentioned in \cite[Section 7.4]{DHKM20}, it may be advantageous to consider a definition of cones that modifies the definition 
\eqref{eq:cone_simple}, in order to keep better track of the decay of the 
quantities $(f,u_{\bsk})_{\calF}$. This would help in limiting the cost of adaptive algorithms in situations where the $(f,u_{\bsk})_{\calF}$ decay fast.
Indeed, consider a partition of $\mathbb{I}$ into a countable number of pairwise disjoint subsets 
$(\mathbb{I}_j)_{j\ge 0}$, i.e., $\mathbb{I}=\bigcup_{j=0}^\infty \mathbb{I}_j$. 
For $j\ge 1$, define
\[
 T_j \colon \mathcal F \to \mathcal F, \quad T_j(f):= \sum_{\bsk\in\mathbb{I}_j} (f,u_{\bsk})_\calF \cdot u_{\bsk}.
\]
That is, the $T_j$ are the basis projections corresponding to the finite index sets $\mathbb{I}_j$.

Then, as in \cite[Section 7.4]{DHKM20}, one can modify \eqref{eq:cone_simple} to
\begin{equation*}
 \widetilde{\mathcal{C}}:= \left\{ f \in \mathcal F \colon \Vert T_{j+\ell} f \Vert_{\mathcal F} \le t v^{\ell} \Vert T_j f \Vert_{\mathcal F}, \ \forall j,\ell\in\NN \right\},
\end{equation*}
where $v<1<t$. 
Opposed to the cones considered in \eqref{eq:cone_simple}, which were defined via a single reference layer, let us call $\widetilde \calC$ a multi-layer cone.

One quickly realizes (using the triangle inequality) that such multi-layer cones are actually contained in a single-layer cone as in \eqref{eq:cone_simple} with $T$ being the basis projection for coefficients $\bsk\in \mathbb{I}_0 \cup \mathbb{I}_1$ and suitably chosen inflation factor.
As such, solvability statements and upper bounds carry over from Corollary~\ref{cor:cost-on-cone}.

We leave it for future research to improve upon the bounds resulting from this inclusion. Noting that the information class $\Lambda^{\rm all}$ has been dealt with in \cite{DHKM20}, 
this problem is especially interesting for the case of information from the class $\Lambda^{\rm std}$.

\medskip

\textbf{Acknowledgement.} \ We would like to thank Aicke Hinrichs and Erich Novak for fruitful discussions on the optimality of homogeneous algorithms.
We also thank Robert Kunsch and three anonymous referees for helpful feedback. David Krieg is supported by the Austrian Science Fund (FWF) Project M~3212-N. Peter Kritzer acknowledges the support of the Austrian Science Fund (FWF) Project P34808.
For open access purposes, the authors have applied a CC BY public copyright license to any author accepted manuscript version arising from this submission. Research work on this paper was partly carried out during the authors' stays
at the MATRIX Institute in Creswick, Australia, during the program ``Computational Mathematics for High-Dimensional Data in Statistical Learning'' (Feb. 2023), and at Schloss Dagstuhl, Germany, during the
Dagstuhl Seminar 23351 ``Algorithms and Complexity for Continuous Problems'' (Aug.~2023); we thank both institutions for their hospitality.


\begin{small}

	\noindent\textbf{Authors' addresses:}\\

        \noindent David Krieg\\
        Department of Analysis\\
	Johannes Kepler University Linz\\
	Altenbergerstr. 69, 4040 Linz, Austria\\
	\texttt{david.krieg@jku.at}\\
 
	\noindent Peter Kritzer\\
	Johann Radon Institute for Computational and Applied Mathematics (RICAM)\\
	Austrian Academy of Sciences\\
	Altenbergerstr. 69, 4040 Linz, Austria\\
	\texttt{peter.kritzer@oeaw.ac.at}\\

\end{small}


\begin{thebibliography}{99}

\bibitem{Bak71} N.\,S.~Bakhvalov. On the optimality of linear methods for operator approximation
in convex classes of functions (in Russian), 
Zh. Vychisl. Mat. Mat. Fiz. 11, 244--249, 1971.

\bibitem{BG} R.~Bartle and L.~Graves. Mappings between function spaces. Trans.\ Am.\ Math.\ Soc.~72, 400--413, 1952.

\bibitem{BKUV} G.~Byrenheid, L.~Kämmerer, T.~Ullrich, T.~Volkmer. Tight error bounds for rank-1
lattice sampling in spaces of hybrid mixed smoothness. Numer.\ Math., 136, 993–
1034, 2017.

\bibitem{CW04} J.~Creutzig and P.~Wojtaszczyk. Linear vs.\ nonlinear algorithms for linear
problems, J.~Complexity~20, 807--820, 2004.

\bibitem{DKP22} J.~Dick, P.~Kritzer, F.~Pillichshammer. \textit{Lattice Rules}. Springer, Cham, 2022. 

\bibitem{DHKM20} Y.~Ding, F.J.~Hickernell, P.~Kritzer, S.~Mak. Adaptive approximation 
for multivariate linear problems with inputs lying in a cone. 
In: F.J.~Hickernell, P.~Kritzer (eds.) \textit{Multivariate Algorithms and Information-Based Complexity}, 109--145, DeGruyter, Berlin/Boston, 2020.

\bibitem{DKU} M.\,Dolbeault, D.\,Krieg, and M.\,Ullrich.  A sharp upper bound for sampling numbers in $L_2$. 
{Appl. Comput. Harmon. Anal.} 63, 
113--134, 2023.

\bibitem{Don94} D.~L.~Donoho. Statistical estimation and optimal recovery. The Annals of Statistics, 22(1), 238--270, 1994.

\bibitem{DPW} R.~DeVore, G.~Petrova,
P.~Wojtaszczyk. Data assimilation and sampling in Banach spaces. Calcolo, 54:963--1007, 2017.

\bibitem{FP23} S.~Foucart, G.~Paouris. 
Near-optimal estimation of linear functionals with log-concave observation errors.
Inf.~Inference~12(4), 2546--2561, 2023.


\bibitem{FPRU}
S.~Foucart, A.~Pajor, H.~Rauhut, and T.~Ullrich. The Gelfand widths of $\ell_p$-balls for $0<p\le 1$, J.~Complexity~26, 629--640, 2010.

\bibitem{FR} S.~Foucart, H.~Rauhut.
\textit{A Mathematical Introduction to Compressive Sensing}.
Birkhäuser, New York, 2013.

\bibitem{GGM23} P.~Gaillard, S.~Gerchinovitz, E.~de Montbrun. Adaptive approximation of monotone functions. \href{https://arxiv.org/abs/2309.07530}{arXiv:2309.07530}, 2023.

\bibitem{GM80} S.~Gal and C.~A.~Micchelli. Optimal sequential and non-sequential procedures for evaluating
a functional. Appl. Anal. 10, 105--120, 1980.

\bibitem{GG84} A.\,Yu.~Garnaev and E.\,D.~Gluskin. 
The widths of a Euclidean ball. Soviet Math. Dokl. 30 (1984), 200--204.

\bibitem{GL22} T.~Goda, P.~L'Ecuyer. Construction-free median quasi-Monte Carlo rules for function spaces with unspecified smoothness and general weights. SIAM J. Sci. Comp. 44, 
A2765–-A2788, 2022.

\bibitem{GSM23} T.~Goda, K.~Suzuki, M.~Matsumoto. A universal median quasi-Monte Carlo integration. 
SIAM J. Numer. Anal. 62(1),
533--566, 2024.


\bibitem{Hei23a} S.~Heinrich. Randomized complexity of parametric integration and the role of adaption I. Finite dimensional case. 
J. Complexity 81, 101821, 2024.

\bibitem{Hei23b} S.~Heinrich. Randomized complexity of parametric integration and the role of adaption II. Sobolev spaces.
J. Complexity 82, 101823, 2024.

\bibitem{Hei24} S.~Heinrich. Randomized complexity of mean computation and the adaption problem. \href{https://arxiv.org/abs/2401.14100}{arXiv:2401.14100}, 2024.


\bibitem{HJR16} F.J.~Hickernell, L.A.~Jim\'{e}nez Rugama. Reliable adaptive cubature using digital sequences. In: R.~Cools, D.~Nuyens (eds.), \textit{Monte Carlo and quasi-Monte Carlo Methods 2014}, 367--383, Springer, Cham, 2016.

\bibitem{HJRL18} F.J.~Hickernell, L.A.~Jim\'{e}nez Rugama, D.~Li. Adaptive quasi-Monte Carlo methods for
cubature. In: J.~Dick, F.Y.~Kuo, H.~Wo\'{z}niakowski (eds.), \textit{Contemporary Computational Mathematics---a celebration of the 80th birthday of Ian Sloan}, 597--619, Springer, Cham, 2018.

\bibitem{K77} B.\,S.~Kashin. Diameters of some finite-dimensional sets and classes of smooth functions.
Math.~USSR Izv.~11, 317–333, 1977.

\bibitem{Kor94} N.\,P.~Korneichuk. 
Optimization of active algorithms for recovery of monotonic functions
from H\"{o}lder’s class. J.~Complexity~10, 265--269, 1994.

\bibitem{KU1} D. Krieg and M. Ullrich. Function values are enough for $L_2$-approximation. 
Found. Comput. Math. 21(4), 1141--1151, 2021.

\bibitem{Ku16} R.\,J.~Kunsch.
Bernstein numbers and lower bounds for the Monte Carlo error.  In: Cools, R., Nuyens, D. (eds), Monte Carlo and quasi-Monte Carlo methods. Springer Proceedings in Mathematics \& Statistics, vol 163, 2016.

\bibitem{KuDiss} R.\,J.~Kunsch.
High-Dimensional Function Approximation: Breaking the Curse with Monte Carlo Methods. Dissertation, Friedrich-Schiller-University, Jena. 
\href{https://arxiv.org/abs/1704.08213}{arXiv:1704.08213}, 2017.

\bibitem{KNR19} R.\,J.~Kunsch, E.~Novak, D.~Rudolf.
Solvable integration problems and optimal sample size selection.
J.~Complexity~53, 40--67, 2019.

\bibitem{KNW} R.\,J.~Kunsch, E.~Novak, M.~Wnuk.
Randomized approximation of summable sequences -- adaptive and non-adaptive.
\href{https://arxiv.org/abs/2308.01705}{arXiv:2308.01705}, 2023.

\bibitem{KSW06} F.Y.~Kuo, I.H.~Sloan, H.~Wo\'{z}niakowski. Lattice rules for multivariate approximation in the worst case setting. In: H.~Niederreiter, D.~Talay (eds.), \textit{Monte Carlo and Quasi-Monte Carlo Methods 2004}, 289--330. Springer, Berlin, 2006.

\bibitem{M90} P.~Math\'e. $s$-Numbers in Information-Based Complexity. J.\ Complexity 6, 41--66, 1990.

\bibitem{NSU} N. Nagel, M. Sch\"afer, and T. Ullrich.
A new upper bound for sampling numbers. 
Found. Comput. Math. 22(2), 445--468, 2022.

\bibitem{N96} E.~Novak. On the power of adaption. J. Complexity 12, 199--237, 1996.

\bibitem{NW08} E.~Novak, H.~Wo\'{z}niakowski. \textit{Tractability of Multivariate Problems. Volume I: Linear Information}. EMS, Z\"{u}rich, 2008.

\bibitem{P86} E.W. Packel. Linear problems (with extended range) have linear optimal algorithms. Aeq. Math. 31, 18--25, 1986. 

\bibitem{Pinkus} A.~Pinkus. \textit{$n$-Widths in Approximation Theory}. Springer, Berlin, 1985.

\bibitem{Pla96} L.~Plaskota. \textit{Noisy Information and Computational Complexity}. Cambridge University
Press, Cambridge, UK, 1996.

\bibitem{S01} K.A.~Sikorski. \textit{Optimal Solution of Nonlinear Euqations}. Oxford University Press, Oxford, UK, 2001.

\bibitem{S60} S.\,A.~Smolyak. Interpolation and quadrature formulas for the classes $W_s^\alpha$ and $E_s^\alpha$. Dokl.\ Akad.\ Nauk SSSR, 131(5), 1028–1031, 1960.

\bibitem{S65} S.\,A.~Smolyak. On optimal restoration of functions and functionals of them
(in Russian), Candidate Dissertation, Moscow State University, 1965.


\bibitem{TW80} J.~F.~Traub, H.~Wo\'zniakowski. \textit{A General Theory of Optimal Algorithms}. Academic Press, NewYork, 1980.

\bibitem{V22} F.~Voigtlaender. $L^p$ sampling numbers for the Fourier-analytic Barron space. Submitted, 2022, \url{https://arxiv.org/abs/2208.07605}.

\bibitem{WW86} A.G. Werschulz, H.~Wo\'{z}niakowski. Are linear algorithms always good for linear problems? Aeq. Math. 31, 202--211, 1986. 

\end{thebibliography}
\end{document}